\documentclass{article}
\usepackage[T2A]{fontenc}
\usepackage[cp1251]{inputenc}
\usepackage[english,russian]{babel}
\usepackage[tbtags]{amsmath}
\usepackage{amsfonts,amssymb,mathrsfs,amscd}

\usepackage[all,arc,poly,2cell,curve,arrow,tips]{xy}
\usepackage[all,arc,poly,2cell,curve,arrow,tips]{xy}

\let\No=\textnumero

\usepackage[hyper]{izv2e}
\JournalName{}

\numberwithin{equation}{section}
\theoremstyle{plain}
\newtheorem{theorem}{Theorem}
\newtheorem{lemma}{Lemma}[section]
\newtheorem{propos}{Proposition}

\newtheorem{cor}{Corollary}

\theoremstyle{definition}
\newtheorem{definition}{Definition}
\newtheorem{proof}{Proof}
\newtheorem{remark}{Remark}

\def\me{\operatorname{me}}

\begin{document}

\title{A functional realization of the  Gelfand-Tsetlin base}
\author[D.\,V.~Artamonov]{D.\,V.~Artamonov}
\address{Lomonosov Moscow State University}
\email{artamonov@econ.msu.ru} 

\udk{ 517.986.68}

\maketitle

\begin{fulltext}

\begin{abstract}

In the paper we consider  a realization of a finite dimensional irreducible representation of the Lie algebra  $\mathfrak{gl}_n$ in the space of functions on the group  $GL_n$.  It is proved that functions corresponding to Gelfand-Tsetlin diagrams  are linear combinations of some new functions of hypergeometric type which are closely related to   $A$-hypergeometric functions. These new functions are solution of a system of partial differential equations  which one  obtains from the  Gelfand-Kapranov-Zelevinsky  by an "antisymmetrization". The coefficients  in the constructed linear combination are hypergeometric constants i.e. they are values of some hypergeometric functions when instead  of all arguments ones are substituted.

Bibliography: 16 items.
\end{abstract}


\markright{The A-GKZ system and the GTs basis}


\section{Inroduction}

In the year  1950 Gelfand and Tsetlin published a short paper   \cite{GC1},  where they gave an indexation of   base vectors in an irreducible finite-dimensional representation of the Lie algebra $\mathfrak{gl}_n$ and  presented formulas for the action of generators of the algebra in this base.  This paper does not contain a derivation of the presented formulas and it was not translated into English. Nevertheless the results of the paper became known in the West and there appeared attempts to  reproduce  the  construction   of the base vectors  and to reprove the  formulas for the action of  the generators. In year  1963 there appeared a paper by Biedenharn and Baird  \cite{1963}, where it was done.



In the Biedenharn's and Baird's  paper  \cite{1963}  in the case  $\mathfrak{gl}_3$  a very interesting derivation of Gelfand and Tsetlin's formulas is given.  Consider a realization of a representation in the space of function on the  group $GL_3$.   Then a function corresponding to Gelfand-Tsetlin base vectors  can be expressed though the Gauss' hypergeometric function  $F_{2,1}$\footnote{In the paper  \cite{1963} a realization of a representation  using the creation and annihilation operators is used, in the case  $\mathfrak{gl}_3$  these realizations are essentially equivalent}.    And the formulas for the action of generators turn out to be consequences of the contiguous relations for this function.

In  \cite{A1}  this approach is used to obtain explicit formulas for the Clebsh-Gordan coefficients for the algebra    $\mathfrak{gl}_3$.  Also this approach is used to obtain an explicit construction of an  infinite-dimensional  representation of  $\mathfrak{gl}_3$ (see \cite{tmph}).  There exist  generalizations of the results of   \cite{1963}  to the case of quantum algebras  (see \cite{t15}, \cite{t16}).  Recently their generalization to the case  $\mathfrak{sp}_4$ were obtained  \cite{ta1}.

In the 60-s it was not possible to obtain a generalization of these construction to the case   $\mathfrak{gl}_n$  since at this time the  theory of multivariate  hypergeometric function was not well-developed. The theory of   $A$-hypergeometric functions did not exist at this time and the system of  equation  for these functions i.e. the Gelfand-Kapranov-Zelevinsky system (the GKZ system for short) was not known. All these objects appeared only in the 80-s of the  $XX$ century.
In the present paper using these results we do generalize the results of Biedenharn and Baird to the case of  general   $n$.


The main result of the present paper is a formula for a function on the group that correspond to a Gelfand-Tsetlin base vector  (Theorems  \ref{gt1},\ref{gt2}).  This result makes possible to give a new derivation of formulas for the action of generators, to obtain formulas for the Clebsh-Gordan coefficients and so on.

The passage form the case $n=3$ to the case  $n>3$ needs new ideas and methods.  In the case  $n=3$ a formula for a function  $f$ corresponding to a Gelfand-Tsetlin diagram is derived using a presentation of   $f$  as a result of application of lowering operators to a highest vector.
 It is not possible to generalize these considerations  in the case $\mathfrak{gl}_3$ to the case $\mathfrak{gl}_n$, $n>3$  since the formulas for the lowering operators $\nabla_{n,k}$  become very complicated (see  \cite{zh}).   Also in the case  $n\geq 4$  there appears a  new difficulty because of the fact that argumets of the function  $f$, which are minors of a matrix,  are not independent, they satify the Plucker relations.

 A possible way to overcome these difficulties is to use ideas form the complex analysis to  find an analogue  in the case  $n\geq 4$ of a function corresponding to a Gelfand-Tsetlin diagram in the case   $n=3$.
This is done in the present paper.  A function of an element  $g\in GL_N$ by analogy with the case  $n=3$,  is  written as a function of minors of $g$.
We note that in the case  $\mathfrak{gl}_3$  the considered function can be written as an  $A$-hypergeometric function.   By analogy with the case  $\mathfrak{gl}_3$ for  $n>3$ we try to find the function of interest as an  $A$-hypergeometric  function of minors. This  function is defined as a sum of a series called a $\Gamma$-series.  A   $\Gamma$-series  is a sum of monomials divided by factorials of exponents.  And  the set of the exponents of monomials in this series  is a shifted lattice  in the space of all possible exponents.  It turns out that it is possible to relate with a Gelfand-Tsetlin diagram  a system of equations that defines a shifted lattice in the space of exponents  (see Section  \ref{agg}). Thus to each Gelfand-Tsetlin  diagram there corresponds a  $\Gamma$-series (which is actually  a finite sum). It is proved that the constructed  functions belong to a canonical embedding of an irreducible finite dimensional  representation  into the functional representation and form a base in it. But this approach does give a solution of the posed problem. Even in the case    $\mathfrak{gl}_4$  the construed functions do not correspond to Gelfand-Tsetlin base vectors. Nevertheless the constructed base is related to the Gelfand-Tsetlin base  by a transformation which is upper-triangular  relatively some order on diagrams  (see  Section  \ref{ra3}).

In order to prove that the constructed  $\Gamma$-series form a base in a representation a new system of PDE is constructed. This system is called the antysimmetrized GKZ system (A-GKZ for short, see Section \ref{agkz}).  We construct a base in the space of it's polynomial solutions. It turns out that there is a bijective correspondence between the constructed   $\Gamma$-series and the constructed basic solutions of the A-GKZ system (Theorem  \ref{agkz}).

Then we show that the constructed base solutions also belong to a canonical embedding of an irreducible finite dimensional  representation of $\mathfrak{gl}_n$ into the functional representation, they form a base in it. This base is related to the Gelfand-Tsetlin base using a low-triangular transformation (see Section \ref{fgc}).   We express a function corresponding to a Gelfand-Tsetlin diagram using these basic solutions.

Since the Gelfand-Tsetlin base is orthogonal relatively an invariant scalar product, the passage form the base consisting of basic solutions of A-GKZ to the Gelfand-Tsetlin base is nothing but the orthogonalization transformation.  To write  this transformation explicitly  one needs to find  scalar products between basic solutions   (see Section  \ref{isk}).  When it is done one finds a low-triangular change of coordinates that diagonalizes the bilinear form of the considered scalar product  (see Section \ref{itg}).
Finally when one has the diagonalizing  change of coordinates one can find the corresponding orthogonal base which is nothing but the Gelfand-Tsetlin base (see Theorems  \ref{gt1},\ref{gt2}).

In these Theorems the functions corresponding to the  Gelfand-Tsetlin base vectors are expressed through the basic solutions of the  A-GKZ system using the numeric coefficients, which are written as sums of some series. In Section  \ref{cft}  we  try to convenience the reader that the obtained formula for the function corresponding to the Gelfand-Tsetlin diagrams is good enough.  The basic solutions of A-GKZ are Horn's hypergeometric function.
  And in Section \ref{cft} the coefficients occurring in Theorems   \ref{gt1},\ref{gt2} are discussed. It is shown that they are hypergeometric constants i.e. values of generalized hypergeometric functions  (in the Horn's sense)  when one substitutes ones instead of all their arguments.  And these generalized hypergeometric functions are expressed  (see \eqref{hrn})  through the Horn functions associated with the  $A$-hypergeometric function constructed in  Section   \ref{agg}.

\begin{remark}
One can overcome some difficulties caused by the Plucker relations by considering as in   \cite{1963} a realization using the creation and annihilation operators (or a realization based   Weyl construction). But the usage of depended minors has some fundamental  advantages. For example in this case one has  a simple description of the space of functions forming an irreducible representation   \cite{zh},  and also the presence of the relations  suggests some fundamental steps in construction of the functions of interest. 
\end{remark}

\section{ Preliminary facts}
\label{ra1}

In this Section some basic objects and construction are introduced. The Theorem for the case   $\mathfrak{gl}_3$ is formulated. This result was a starting point for the present paper.

\subsection{A realization in the space of functions on a group}
In the paper Lie groups and algebras over the field $\mathbb{C}$ are considered.

Function on the group  $GL_n$  form a representation of the group  $GL_n$. An element  $X\in GL_{n}$  acts onto a function $f(g)$, $g\in GL_n$   by a right shift

\begin{equation}
\label{xf} (Xf)(g)=f(gX).
\end{equation}

Passing to an infinitesimal version of this action one obtains an action of the Lie algebra  $\mathfrak{gl}_n$ on the space of all functions.

Every irreducible finite-dimensional representation can be realized as a sub-representation in the space of functions.
Let 
$[m_{1},...,m_{n}=0]$  be a highest weight, then in the space of functions there exists a highest vector with such a weight which is written as follows.

Let $a_{i}^{j}$, $i,j=1,...,n$  be a function of a matrix element on the group  $GL_{n}$.  Here $j$ is a row index and   $i$ is a column index.
Also put

\begin{equation}
\label{dete}
a_{i_1,...,i_k}:=det(a_i^j)_{i=i_1,...,i_k}^{j=1,...,k}.
\end{equation}

That is one takes a determinant of a sub-matrix in  a matrix $(a_i^j)$,
formed by rows   $1,...,k$  and columns
$i_1,...,i_k$.  An operator $E_{i,j}$  acts onto this determinant by changing the column indices 

\begin{equation}
\label{edet1}
E_{i,j}a_{i_1,...,i_k}=a_{\{i_1,...,i_k\}\mid_{j\mapsto i}},
\end{equation}

where $.\mid_{j\mapsto i}$ is substitution of   $j$  instead of 
$i$. If the index  $j$ does not occur in   $\{i_1,...,i_k\}$, then one obtains zero.

Using  \eqref{edet1}, one can easily see that the vector

\begin{equation}
\label{stv}
v_0=\frac{a_{1}^{m_{1}-m_{2}}}{(m_1-m_2)!}\frac{a_{1,2}^{m_{2}-m_{3}}}{(m_2-m_3)!}...\frac{a_{1,2,...,n-1}^{m_{n-1}}}{m_{n-1}!}
\end{equation}

is a highest vector for the algebra $\mathfrak{gl}_{n}$ with the weight
$[m_{1},m_{2},...,m_{n-1},0]$.  Thus one has a canonical embedding of an irreducible finite-dimensional representation  into the functional representation.

If one considered a highest weight with a non-zero component  $m_n$  then in all formulas below one must change   $m_{n-1}\mapsto m_{n-1}-m_{n}$ and multiply all expressions by  $a_{1,2,...,n}^{m_n}$. To make formulas less cumbersome we put $m_n=0$.

\subsection{The Gelfand-Tsetlin base}

Consider a chain of subalgebra $\mathfrak{gl}_n\supset\mathfrak{gl}_{n-1}\supset...\supset \mathfrak{gl}_1$.  Let us be given an irreducible representation  $V_{\mu_n }$ of the algebra $\mathfrak{gl}_n$ with the highest weight $\mu_n$.  When one restricts the algebra  $\mathfrak{gl}_n\downarrow \mathfrak{gl}_{n-1}$  the representation  ceases to be irreducible and it splits into a direct sum of irreducible representations of  $\mathfrak{gl}_{n-1}$. Every irreducible representation of  $\mathfrak{gl}_{n-1}$ can occur in this direct sum with multiplicity  not greater than one \cite{zh}.  Thus one has

 $$
 V_{\mu_n}=\sum_{\mu_{n-1}}  V_{\mu_n;\mu_{n-1}},
 $$

where   $\mu_{n-1}$ are possible  $\mathfrak{gl}_{n-1}$-highest weight and   $ V_{\mu_n;\mu_{n-1}}$ is a representation of  $\mathfrak{gl}_{n-1}$ with the highest weight   $\mu_{n-1}$.   The sum is taken over all  $\mathfrak{gl}_{n-1}$-highest weights occurring in the decomposition of $V_{\mu_n }$  into irreducible representations.

When one continuous restrictions $\mathfrak{gl}_{n-1}\downarrow \mathfrak{gl}_{n-2}$  and so on one gets a splitting of the following type 

\begin{equation}
\label{rmm}
V=\sum_{\mu_{n-1}} \sum_{\mu_{n-2}}...  \sum_{\mu_1}V_{\mu_n;...;\mu_1}.
\end{equation}

Here  $V_{\mu_n;...;\mu_1}$ is an irreducible representation of the algebra   $\mathfrak{gl}_1$, thus  $dimV_{\mu_n;...;\mu_1}=1$.   When one chooses  base vectors in all $V_{\mu_n;...;\mu_1}$ one obtains a base in   $V_{\mu_n}$. This is the Gelfand-Tsetlin base.

The base vectors are encoded by sets of highest weights $(\mu_n;...;\mu_1)$, appearing in splitting \eqref{rmm}.  If one writes these highest weights one under another in the following way 

$$
\begin{pmatrix}
m_{1,n}&& m_{2,n}&&...&& m_{n,n}\\
& m_{1,n-1}&& m_{2,n-1}&...& m_{n-1,n-1}\\
&&&&...\\
&&&&m_{1,1}
\end{pmatrix},
$$

then one obtains a diagram that is called   {\it the Gelfand-Tsetlin diagram}. We denote it as   $(m_{i,j})$.  For elements of this diagram the betweenness condition holds: if a element of a row  occurs between two elements of an upper row  then it lies between them in the numeric sense.  The contra verse is also true: every diagram from which the betweenness condition holds appears as a Gelfand-Tsetlin  diagram in the splitting \eqref{rmm}.

\subsection{  $A$-hypergeometric functions}

\subsubsection{A $\Gamma$-series}  One can find information about   a $\Gamma$-series in  \cite{GG}.

Let $B\subset \mathbb{Z}^N$  be a lattice, let  $\gamma\in \mathbb{Z}^N$ be a fixed vector. Define a  {\it  hypergeometric
$\Gamma$-series  } in variables $z_1,...,z_N$ by the formula

\begin{equation}
\mathcal{F}_{\gamma}(z)=\sum_{b\in
B}\frac{z^{b+\gamma}}{\Gamma(b+\gamma+1)},
\end{equation}
where  $z=(z_1,...,z_N)$. In the numerator and in the denominator the multi-index notations are used 

$$
z^{b+\gamma}:=\prod_{i=1}^N
z_i^{b_i+\gamma_i},\,\,\,\Gamma(b+\gamma+1):=\prod_{i=1}^N\Gamma(b_i+\gamma_i+1).
$$

For the  $\Gamma$-series considered in the present paper the vectors of exponents   $b+\gamma$ have integer coordinated. In this case instead of   $\Gamma$-functions it is reasonable to use shorter notations with factorials.  Hence in the denominator instead of  $\Gamma(b+\gamma+1)$ we shall write

$$
(b+\gamma)!:=\prod_{i=1}^N(b_i+\gamma_i)!
$$

We shall use an agreement that a factorial of a negative integer equals to infinity.

We need the following properties of  $\Gamma$-series:

\begin{enumerate}
\item A vector  $\gamma$ can be changed to $\gamma+b$, $b\in B$,  the series remains unchanged. 

\item $\frac{\partial }{\partial
z_i}\mathcal{F}_{\gamma,B}(z)=\mathcal{F}_{\gamma-e_i,B}(z)$, где $e_i=(0,...,1_{\text{
на месте }i},...,0)$.

\item
Let  $F_{2,1}(a_1,a_2,b_1;z)=\sum_{n\in\mathbb{Z}^{\geq
0}}\frac{(a_1)_n(a_2)_n}{(b_1)_n}z^n$, where 
$(a)_n=\frac{\Gamma(a+n)}{\Gamma(a)}$ is the Gauss hypergeometric series. Then for
 $\gamma=(-a_1,-a_2,b_1-1,0)$  and
$B=\mathbb{Z}<(-1,-1,1,1)>$ one has

\begin{align*}
&\mathcal{F}_{\gamma}(z_1,z_2,z_3,z_4)=cz_1^{-a_1}z_2^{-a_2}z_3^{b_1-1}F_{2,1}(a_1,a_2,b_1;\frac{z_3z_4}{z_1z_2})\\
&c=\frac{1}{\Gamma(1-a_1)\Gamma(1-a_2)\Gamma(b_1)}.
\end{align*}

A sum of a   $\Gamma$-series  (when it converges) is called an  $A$-hypergeometric function.

\subsubsection{ The Gelfand-Kapranov-Zelevinsky system}

An $A$-hypergeometric function satisfies a system of partial differential equation which consists of  equations of two types.

{\bf 1.} Let  $a=(a_1,...,a_N)$ be a vector orthogonal to the lattice $B$,  then

\begin{equation}
\label{e1}
a_1z_1\frac{\partial}{\partial z_1}\mathcal{F}_{\gamma}+...+a_Nz_N\frac{\partial}{\partial z_N}\mathcal{F}_{\gamma}=(a_1\gamma_1+...+a_N\gamma_N)\mathcal{F}_{\gamma},
\end{equation}
it is enough to consider only the base vectors of an orthogonal complement to $B$.

{\bf 2.} Let  $b\in B$ and $b=b_+-b_-$, where all coordinates of  $b_+$, $b_-$ are non-negative.  Let us emphasize non-zero elements in these vectors
$b_+=(...b_{i_1},....,b_{i_k}...)$,  $b_-=(...b_{j_1},....,b_{j_l}...)$. Then

\begin{equation}
\label{e2} (\frac{\partial }{\partial
	z_{i_1}})^{b_{i_1}}...(\frac{\partial}{\partial z_{i_k}})^{b_{i_k}}
\mathcal{F}_{\gamma}=(\frac{\partial }{\partial
	z_{j_1}})^{b_{j_1}}...(\frac{\partial }{\partial z_{j_l}})^{b_{j_l}} \mathcal{F}_{\gamma}.
\end{equation}

It is enough to consider only the base vectors $b\in B$.
\end{enumerate}



\subsection{The case of $\mathfrak{gl}_3$}
\label{gl3tm}

The results presented in this Section were obtained in    \cite{1963},  from a modern point of view they are reformulated in \cite{a1}.

A Gelfand-Tsetlin diagram  in the case  $\mathfrak{gl}_3$ can be written as follows

\begin{equation}
	\label{dgc3}
\begin{pmatrix}
m_{1} && m_{2} &&0\\ &k_{1}&& k_{2}\\&&h_{1}
\end{pmatrix},
\end{equation}

Let us give a formula for the function corresponding to the  Gelfand-Tsetlin diagram  \eqref{dgc3}.  

\begin{theorem}\label{vec3} 
	
  Put determinants \eqref{dete} in the following order  $$a=(a_{1},a_{2},a_{3},a_{1,2},a_{1,3},a_{2,3}),$$
	
	take  a lattice
	
	$$
B=\mathbb{Z}<	(1,-1,0,0,-1,1)>.
	$$

	 $\gamma=(h_{1}-m_{2},k_{1}-h_{1},    m_{1}-k_{1}        ,        k_{2}    ,m_{2}-k_{2},0)$.  Then to the diagram there corresponds a functions   $\mathcal{F}_{\gamma}(a)$.
\end{theorem}


This  $\Gamma$-series can be expressed through the Gauss hypergeometric function  (see property 3 for the   $\Gamma$-series).   In this form the Theorem was obtained in \cite{1963}.  One can easily see that the considered  $\Gamma$-series is a finite sum.

Note that the lattice  $B$  can be defined by equations with the following conditions for sums of exponents of determinant that contain some certain indices :

\begin{equation}
\label{ssy}
\begin{cases}
\text{  The sum of exponents of determinants that contain $1$, or $2$,  or $2$}=m_{1},\\
\text{ The sum of exponents of determinants that contain $1$ and $2$, $1$ and $3$, $2$ and $3$}=m_{2},\\
\text{ The sum of exponents of determinants that contain $1$  or $2$ }=k_{1},\\
\text{ The sum of exponents of determinants that contain $1$ and $2$ }=k_{2},\\
\text{ The sum of exponents of determinants that contain $1$}=h_{1}.
\end{cases}
\end{equation}

These equations have a natural explanation form the point of view of the representation theory.  Take a function of determinants that corresponds to a Gelfand-Tsetlin diagram, decompose it into a sum of products of determinants.

The last equation guarantees that this function is a weight vector for the operator $E_{1,1}$ with the weight  $m_{1,1}$.

The third and the fourth equations guarantee that after an application of the raising operator  from  $\mathfrak{gl}_2$   (in the needed power)  one  obtains a highest vector for the algebra $\mathfrak{gl}_2$ with the highest weight  $[k_1,k_2]$.

 The first and the second conditions guarantee that after an application of the raising operator  from   $\mathfrak{gl}_3$   (in the needed power)  one  obtains a highest vector for the algebra   $\mathfrak{gl}_3$ with the highest weight  $[m_1,m_2,0]$.

Because of the presence of relations between determinants these conditions  do not necessarily hold for a function corresponding summands of the functions corresponding to a  Gelfand-Tsetlin diagram. But Theorem  \ref{vec3} states that these elementary considerations do give a right answer in the case   $\mathfrak{gl}_3$.

\section{The Gelfand-Tsetlin lattice}

Now we proceed to  finding of an analogue of the Theorem \ref{vec3} in the case  $n>3$.
Let first write an analogue of the equations  \eqref{ssy}  in the case  $n>3$ and let us investigate the obtained lattice.


Consider a complex linear space whose coordinates are indexed by proper subsets  $X\subset\{1,...,n\}$.
Denote the coordinates as  $z_X$, $X=\{i_1,...,i_l\}$. Denote a corresponding unit vector as   $e_{X}$.

\begin{definition} Let   $B_{}$ be a lattice defined by the following equations.  Fix  $1\leq p\leq q\leq n$.
		A vector
	$v\in  B_{}$ if and only if the sum of it's coordinates, that contain at least    $p$ indices such that they are   $\leq q$, equals to $0$:
	
	$$
v\in B_{} \Leftrightarrow 	\sum_{X:  \text{ X contains}\geq  p \text{ indices such that }  \leq q} v_X=0.
	$$
\end{definition}

{\it Everywhere below  $B$ denotes this lattice.}

Take in   $\mathbb{Z}^N$ a standard scalar product:

$$
<x,y>:=\sum_{i=1}^N x_iy_i,\,\,\, x,y\in \mathbb{Z}^N
$$

Then the definition of    $B$ says that an orthogonal complement to  $B_{}$  is spanned by vectors $\chi_p^q$, defined as follows

\begin{equation}
\label{xii1}
(\chi_p^q)_X=\begin{cases} 1 ,\text{if  X contains }\geq  p \text{ indices such that }  \leq q  \\ 0  \text{ иначе} \end{cases}
\end{equation}


Let us find a base in $B_{}$.



\begin{definition}
	Let $X\subset \{1,...,n\}$. One says that $$s<X,$$ if every element of   $X$ is greater than  $s$. The equality $s<\emptyset$ holds.
\end{definition}




Consider
vectors



\begin{equation}
\label{ijs2}
v_{i,j,x,X}=(...,1_{z_{1,...,i-1,i,X}},...,-1_{z_{1,...,i-1,j,X}},...,-1_{z_{1,...,i-1,i,xX}},...,1_{z_{1,...,i-1,j,xX}},...).
\end{equation}

A base in the lattice  $B$ is constructed as a  subset $\mathcal{I}$  in the set of all vectors   \eqref{ijs2}.   More precise   $\mathcal{I}$  is a subset of vectors of type  \eqref{ijs2},  such that the following conditions hold.  Firstly for all chosen vectors   $v_{i,j,x,X}$ one has  $i<j<x<X$.  
Secondly the following condition holds. For fixed  $i,j$ consider a graph whose vertices correspond to subsets  $X\subset\{j+2,...,n\}$, and  edges correspond to pairs $X,xX $ (not necessary  $x<X$), where we use notation

$$xX:=\{x\}\cup X.$$

 Thus  the edges correspond to vectors  \eqref{ijs2}.   Given a  subset    $\mathcal{I}$  of vectors of type    \eqref{ijs2} there corresponds a graph consisting of vertices defined by vectors form the subset and edges whose ends are these vertices.
The second condition is the following. We claim that  {\bf a) }  the subset  $\mathcal{I}$ is such that for all   $i,j$  the corresponding graph is a tree,  {\bf b) }  the subset    $\mathcal{I}$  is maximal with respect to extensions preserving the property  {\bf a)}. 

\begin{propos}
	
The chose vectors    \eqref{ijs2} form a base in  $B_{}$
\end{propos}

Note that the obtained formulas for the base vectors are very close to the formulas appearing in the construction of hypergeometric functions associated with grassmanians in Plucker coordinates  \cite{GG}.

\begin{proof} {\bf 1.}  Let us prove that the vectors from $\mathcal{I}$   
	 are linearly independent.

Let us begin with  {\it  a general remark}.  Let us be given a set of linearly independent vectors   $\{v_{\alpha}\}$,   then the set consisting of  differences  $\{v_{\alpha_p,\alpha_q}=v_{\alpha_p}-v_{\alpha_q}\}$ is linearly independent if and only if  it does not contain 
{\it cyclic } subsets $v_{\alpha_{p_1},\alpha_{p_2}},v_{\alpha_{p_2},\alpha_{p_3}},...,v_{\alpha_{p_n},\alpha_{p_1}}$.  The corresponding vanishing non-trivial linear combination is just a sum of these vectors.

Now consider the vectors of type

\begin{equation}
\label{vijx}
v_{i,j;X}=(...,1_{z_{i}},...,-1_{z_{j}},...,-1_{z_{iX}},...,1_{z_{jX}},...).
\end{equation}

Let us  find when a set of vectors of such type is linearly independent.  Let us apply a projection of these vectors onto  the subspace spanned by coordinates of type  $z_{i}$. Consider  the obtained vectors
$$
v_{i,j}=(...,1_{z_{i}},...,-1_{z_{j}},...).
$$

Let us apply {\it the general remark } to the subset of vectors $\{v_{\alpha}\}=\{e_{z_{i}}\}$. One obtains  that the set of vectors  $v_{i,j}$  is linearly independent if and only if it contains a subset of type 

$$
v_{i_1,i_2},v_{i_2,i_3},...,v_{i_n,i_1}.
$$

The corresponding vanishing  non-trivial linear combination is proportional to the sum of these vectors.
The sum of projections of the vectors  $
(...,1_{z_{i}},...,-1_{z_{j}},...,-1_{z_{iX}},...,1_{z_{jX}},...)
$  onto other coordinates can  vanish only in the case when all these vectors have the same set   $X$.
If there are no such subsets then the set of vectors  \eqref{vijx} is linearly independent. In particular  in the set  vectors of type  $\{v_{i,j;\{1,...,i-1,X\}}\}$   there are no such subsets. Thus it is linearly independent.


Now let us apply  {\it the general remark} to the set $\{v_{\alpha}\} =\{v_{i,j;\{1,...,i-1,X\}}\}$.  One obtains that this set is linearly independent.

The vectors of type  \eqref{ijs2}
are vectors of type $v_{\alpha_i,\alpha_j}$ for the following set $\{v_{\alpha}\}$: \begin{equation}\label{razn}v_{i,j,x,X}=v_{i,j;\{1,...,i-1,X\}}-v_{i,j;\{1,...,i-1,xX\}}.\end{equation}





Thus the set of vectors $\{v_{i,j,x,X}\}$ is linearly dependent if and only if it contains a  {\it cyclic subset}.  Let us be given a cyclic subset consisting of  $q$ vectors.  Then after reordering one has  

$$
i_1=i_2=...=i_q,\,\,\, j_1=j_2=...=j_q.
$$
That is the cyclic subset consists of vectors  $\{v_{i,j,x_l,X_l}\}$, $l=1,...,q$. 
Also one gets that the graph whose vertices correspond to subsets   $X_l$, $x_lX_l$ and    edges correspond to edges  $\overline{X_l, x_lX_l}$ topologically is a circle.

For the vectors   \eqref{ijs2}  such situations are prohibited. Thus these vectors are linearly independent.

{\bf 2.} The vectors $\mathcal{I}$ 
 are orthogonal to all the vectors $\chi_p^q$.

The vectors of type  \eqref{ijs2}  have non-zero coordinates only with indices   $z_{1,...,i-1,i,X}$, $z_{1,..,i-1.,j,X}$, $z_{1,...,i-1,i,xX}$, $z_{1,...,i-1,j,xX}$,  these coordinates are $(1,-1,-1,1)$.   Let us write these coordinates for the vectors  $\chi_p^q$  and let us check that they are orthogonal to   $(1,-1,-1,1)$.

Let   $q\geq j$.
 Then the coordinates  $z_{1,...,i-1,i,X},z_{1,...,i-1,j,X}, z_{1,...,i-1,i,x,X},z_{1,...,i-1,j,x,X}$ of the vector $\chi_p^q$ in the case $p\leq i$ are equal to  $(1,1,1,1)$, thus this vector is orthogonal to  $(1,-1,-1,1)$.  And in the case $p>i$  the considered coordinates of the vector $\chi_p^q$  coincide with the numbers  $(\chi_{p-i}^q)_{z_{X}}$, $(\chi_{p-i}^q)_{z_{X}}$, $(\chi_{p-i}^q)_{z_{x,X}}$, $(\chi_{p-i}^q)_{z_{x,X}}$.  This numbers form a vector which is orthogonal to   $(1,-1,-1,1)$.

Let   $i\leq q<  j$.
    Then the coordinates  $z_{1,...,i-1,i,X},z_{1,...,i-1,j,X}, z_{1,...,i-1,i,x,X},z_{1,...,i-1,j,x,X}$ of the vector  $\chi_p^q$ in the case $p\leq i-1$ are equal to  $(1,1,1,1)$, and in the case $p> i-1$  these coordinates of the vector  $\chi_p^q$  coincide with the numbers $(\chi_{p-i+1}^q)_{z_{i}}$, $(\chi_{p-i+1}^q)_{z_{i+1}}$, $(\chi_{p-i+1}^q)_{z_{i,x}}$, $(\chi_{p-i+1}^q)_{z_{i+1,x}}$. In both cases they are orthogonal to the vector   $(1,-1,-1,1)$.

Finally in the case  $q\leq i-1$ the coordinates   $z_{1,...,i-1,i,X},z_{1,...,i-1,j,X}, z_{1,...,i-1,i,x,X},z_{1,...,i-1,j,x,X}$ of the vector $\chi_p^q$ are equal to  $(1,1,1,1)$.  This numbers form a vector which is orthogonal to   $(1,-1,-1,1)$.

{\bf 3.}    The number of vectors in $\mathcal{I}$ 
 equals to the rank of the lattice  $B_{}$.  Indeed,  $B_{}$  is embedded into the space of dimension  $2^n-1$.  The lattice is defined by $n+(n-1)+...+1$ equations. Thus the rank of   $B_{}$ equals $2^n-1-(1+...+n)$.

The number of vectors  \eqref{ijs2} can be found in the following manner.  Note that the number of edges in a tree equals to the number of  vertices minus $1$.  Thus the number of vectors  \eqref{ijs2}  with fixed   $i$ equals to the number of choices of  $j$, $x$ and  $X$.  Thus the number of  vectors\eqref{ijs2}
 equals to the sum by $i$ from  $1$  to $n-2$ of the number of subsets in  $\{i+1,...,n\}$, that consist of at least two elements:

\begin{align*}
&\sum_{i=1}^{n-2}(2^{n-i}-1-(n-i))=\sum_{t=2}^{n-1}(2^t-1-t)=2^n-1-(1+...+n).
\end{align*}
 The Proposition is proved.

\end{proof}






Below in the paper we use an index   $\alpha$  and a number $k$,  let us give their definitions.

\begin{definition}
	\label{ia}
 Chose in the lattice  $B_{}\subset \mathbb{Z}^N$ a base $\mathcal{I}$. Then  $k$  is the number of vectors in the base and   $\alpha$  is an index running through the set $\{1,...,k\}$.
\end{definition}

The choice of the base is suggested to be fixed in the remaining part of the paper.


If a base vector is written as follows 

$$v_{i,j,x,X}=(...,1_{{iX}},...,-1_{{jX}},...,-1_{{ixX}},...,1_{{jxX}},...),$$

then introduce notations


\begin{align}
\begin{split}
\label{vpm}
&v^{\alpha}_+=e_{iX}^{}+e_{jyX}^{},\\
&v^{\alpha}_-=e^{}_{jX}+e^{}_{iyX},\\
&v^{\alpha}_0=e^{}_{yX}+e^{}_{ijX},\\
\end{split}
\end{align}

Then

$$
v^{\alpha}=v^{\alpha}_+-v^{\alpha}_-.
$$

With each base vector $v^{\alpha}$   one associates the vector
\begin{equation}
\label{ra}
r^{\alpha}:=v^{\alpha}_0-v^{\alpha}_+.
\end{equation}

\section{Functions associated with the lattice }


In this section we introduce some functions of hypergeometric type associated with the Gelfand-Tsetlin lattice  $B$.

\subsection{  $A$-hypergeometric functions}

\label{agg}

\subsubsection{  A $\Gamma$-series associated with the lattice}

Consider a Gelfand-Tsetlin diagram   $(m_{i,j})$. To this diagram there corresponds a shifted lattice    $\gamma+B$,  defined by the following  equations.

\begin{definition}\label{sdr}
A vector
	$x\in  \mathbb{Z}^N$ belongs to   $\gamma+B$  if and only if for all $1\leq p\leq q\leq n$  the sum of coordinates indexed by such sets  $X$,  that contain at least  $p$  indices satisfying   $\leq q$,  equals to $ m_{p,q}$.
\end{definition}
The shift vector  $\gamma$  is not uniquely defined.  Below we introduce functions    $\mathcal{F}_{\gamma}$ that do not depend on the choice of  $\gamma$ and the functions $J_{\gamma}^s$, $F_{\gamma}$ that do depend on the choice of  $\gamma$.

To a shifted lattice there corresponds an  $A$-hypergeometric function written as a   $\Gamma$-series:
 \begin{equation}
 \label{gmr}
\mathcal{F}_{\gamma}(z)=\sum_{x\in \gamma+B}\frac{z^x}{x!}=\sum_{t\in\mathbb{Z}^k}\frac{z^{\gamma+tv}}{(\gamma+tv)!},
\end{equation}
where we use a fixed base    $\mathcal{I}=\{v^1,...,v^k\}$   in the lattice  $B$  and we use a notation  $tv:=t_1v^1+...+t_kv^k$.  The second expression for  $\mathcal{F}_{\gamma}(z)$ in  \eqref{gmr} is used below to define more general  functions  $J^s_{\gamma}(z)$ (see  \eqref{fs}).

Due to the properties of a $\Gamma$-series,  the expression \eqref{gmr}  does not  depend on the choice of    $\gamma$  $mod B$.

Note that the sum in the expression  \eqref{gmr}  is finite. Indeed the base vectors  $v^1,...,v^k$ have the following property:  they have both positive and negative coordinates.  Thus  only for finite number of  $t\in\mathbb{Z}^k$ a vector of exponents  $\gamma+tv$ has  only non-negative coordinates.  Since we divide by factorials of exponents in  \eqref{gmr}  only such summands are non-zero.

 Also note that in the case when a diagram does not satisfy the betweenness condition  then  \eqref{gmr} equals  $0$.

The constructed functions  $\mathcal{F}_{\gamma} $,  into which one substitutes  the determinants $a_X$,  are direct analogues in the case  $\mathfrak{gl}_n$, $n>4$  of the functions  that in the case    $\mathfrak{gl}_3$ correspond to the Gelfand-Tsetlin diagrams.  Thus it is natural to  ask the question whether the same is true in the case   $\mathfrak{gl}_n$, $n>4$.
Unfortunately this hypothesis is not true \footnote{To prove it one needs to investigate an action of  $E_{4,3}$ onto  $\mathcal{F}_{\gamma}$.  One finds that the selection rulers for this action differ from the  selection rulers for the Gelfand-Tsetlin base.}. Nevertheless  the construction of   $\mathcal{F}_{\gamma} $ is an important step to the construction of functions corresponding to the Gelfand-Tsetlin diagrams.



 \subsubsection{Why a   $\Gamma$-series?}

Let us give arguments (not only an analogy with the case  $n=3$), why it is natural to begin a search for a function corresponding to a Gelfand-Tsetlin diagram  with a construction of a  $\Gamma$-series.

Let us present this function as a sum of products of determinants. Which  product can occur in this sum?   Naive considerations at the end of Section    \ref{gl3tm}  show that it is natural to expect that the exponents of these products satisfy the system form the definition  \ref{sdr}. A $\Gamma$-series is the simplest function for which the exponents form a set defined by the system from the definition  \ref{sdr}.

Mention that there exists a construction of hypergeometric functions associated with homogeneous spaces   \cite{GKZ}.  Since a representation can be realized as a space of section of a bundle over a flag variety \cite{flag},  it is natural to ask a question:  is it useful to take the hypergeometric functions associated with flag varieties, do they coincide with the functions construed in the present paper?   It turns out  that the hypergeometric functions associated with flag varieties are not related with the functions  construed in the present paper.   The  hypergeometric functions associated with flag varieties  can be considered as functions of Plucker coordinates. The Plucker coordinates are coordinates in the space  $\mathbb{C}^N$, indexed by proper subsets  $X\subset \{1,...,n\}$. The  hypergeometric functions associated with flag varieties then are  related with  $A$-hypergeometric functions that are defined by a  lattice  $E$ в in the space  $\mathbb{C}^N$,  which  a lattice of characters  of a torus action on    $\mathbb{C}^N$,  which is obtained from a standard action of an  $n$-dimensional torus on  $\mathbb{C}^n$  \cite{GG}.  But the lattice  $B$  does not  coincide with  $E$, since the corresponding torus action on   $\mathbb{C}^N$ {\it can not  }  be obtained from a standard action of the  $n$-dimensional torus onto $\mathbb{C}^n$.   One can show that  $B \varsubsetneq E$.
 Form this inclusion it follows that the dimension of the space of hypergeometric functions associated with a flag variety is lower than the dimension of a representation.

 Note also the following curiosity. With the lattice    $B$ one can associate an action of a torus  (for which this lattice is a lattice of characters) of dimension $n+(n-1)+...+1$ (see  \cite{GKZ}).  There exists another construction of an action of this torus called the Gelfand-Tsetlin action   \cite{gtac}, but it has nothing to do with the problems considered in the present paper since this is an action on another space.

\subsection{A function $J_{\gamma}^s$}
We need functions that are hypergeometric in the Horn's sense  (see \cite{sts}),  which are generalizations of  a $\Gamma$-series.

Remind that  $k=rkB_{}$,  put

\begin{equation}
\label{fs}
J_{\gamma}^s(z):=\sum_{t\in\mathbb{Z}^k}\frac{(t+1)...(t+s)}{(\gamma+tv)!}z^{\gamma+tv}.
\end{equation}

The multi-index notation is used:  $(t+1)...(t+s):=\prod_{\alpha=1}^k(t_{\alpha}+1)...(t_{\alpha}+s_{\alpha})$.

Note that these functions depend on the choice of  the  shift vector  $\gamma$ in the class $\gamma\, mod \,B_{}$.

In the lattice  $B$   the base  \eqref{ijs2} is fixed and a notation  $\alpha$ for the index  numeration the base vectors was introduced (see definition  \ref{ia}).   With each base vector $v^{\alpha}=v_{i,j,x,X}$ of type     \eqref{ijs2} one associates a differential operator  $\mathcal{O}_{\alpha}$,  called the GKZ operator:

$$
\mathcal{O}_{\alpha}:= \frac{\partial^2 }{\partial z_{1,...,i-1,i ,X}\partial    z_{1,...,i-1,j,xX}}- \frac{\partial^2 }{\partial z_{1,...,i-1,j ,X}\partial    z_{1,...,i-1,i,xX}}.
$$

This differential operator appears in the equation  \eqref{e2}  of the GKZ system associated with the lattice  $B$.

\begin{lemma}
	For   $\alpha\in\{1,...,k\}$  one has
	$$
	\mathcal{O}_{\alpha}J_{\gamma}^s(z)=s_{\alpha}J_{\gamma-v^{\alpha}_+}^{s-e_{\alpha}}(z).
	$$
\end{lemma}

\proof

Denote $(t+1)...(t+s)$ as  $c_t$.    Using the multi-index notation one can write shortly   $\mathcal{O}_{\alpha}=\frac{\partial^2}{\partial z^{v^{\alpha}_+}}-\frac{\partial^2}{\partial z^{v^{\alpha}_-}}$. Then

\begin{align*}
&\frac{\partial^2}{\partial z^{v^{\alpha}_+}}J_{\gamma}^s(z)=\sum_{t\in\mathbb{Z}^k}\frac{c_tz^{\gamma+tv-v^{\alpha}_+}}{(\gamma+tv-v^{\alpha}_+)!}.
\end{align*}

\begin{align*}
&\frac{\partial^2}{\partial z^{v^{\alpha}_-}}J_{\gamma}^s(z)=\sum_{t\in\mathbb{Z}^k}\frac{c_tz^{\gamma+tv-v^{\alpha}_-}}{(\gamma+tv-v^{\alpha}_-)!}=
\sum_{t\in\mathbb{Z}^k}\frac{c_{t-e_{\alpha}}z^{\gamma+tvv^{\alpha}_+}}{(\gamma+tv-v^{\alpha}_+)!},
\end{align*}

where $t-e_{\alpha}=(t_1,...,t_{\alpha}-1,...,t_k)$, where $e_{\alpha}:=(0,....,1_{\text{at the place  } \alpha},...,0)$.
Thus

\begin{align*}
&\mathcal{O}_{\alpha}J_{\gamma}^s(z)=\sum_{t\in\mathbb{Z}^k} \frac{(c_t-c_{t-1^{\alpha}})z^{\gamma+tv-v^{\alpha}_+}}{(\gamma+tv-v^{\alpha}_+)!}.
\end{align*}

But $c_t-c_{t-e_{\alpha}}=(t+1)...(t+s)-t...(t+s-e_{\alpha})=s_{\alpha}(t+1)...(t+s-e_{\alpha})$.

\endproof

Note an analogy

$$
\frac{d}{dz_{{\alpha}}}=s_{\alpha}z^{s-e_{\alpha}}\leftrightarrow \mathcal{O}_{\alpha} J_{\gamma}^s=s_{\alpha}J_{\delta-v^{\alpha}_+}^{s-e_{\alpha}}.
$$

This analogy makes possible the following construction.  One has 

$$
e^z=\sum_{s\in\mathbb{Z}_{\geq 0}^k}\frac{1}{s!}z^s  \Rightarrow  \frac{d}{dz}e^z=e^z.
$$

Define a hyperexpontent:

$$
e^{J}:=\sum_{s\in\mathbb{Z}_{\geq 0}^k}\frac{1}{s!}J^s_{\delta+sv_+}\Rightarrow  \mathcal{O}_i e^{J}= e^{J}.
$$

Continue this construction. Consider an operator  $\frac{\partial^2}{\partial a^{v_0}}$. Define a modified hyperexponent 

$$
\me^{J}:=\sum_{s\in\mathbb{Z}_{\geq 0}^k}\frac{(-1)^s}{s!}J^s_{\delta-sr},
$$

where $r_{\alpha}=v^{\alpha}_0-v^{\alpha}_+$.
It satisfies the equations 

$$
\mathcal{O}_{\alpha}\me^{J}=-\frac{\partial^2}{\partial z^{v^{\alpha}_0}}\me^{J}.
$$

Note that since function of complex variables are considered, the operator    $ \mathcal{O}_{\alpha} $  is essentially the Laplace operator and  $\mathcal{O}_{\alpha}+\frac{\partial^2}{\partial z^{v_0}}$  is a wave operator.











\section{The antisymmetrized  GKZ system}

\label{agkz}

Introduce  the antisymmetrized  GKZ system which is an important instrument in further considerations.  This a system of partial differential equations. In this Section we construct a base in the space of it's polynomial solutions.

We have introduced the GKZ operator:

$$
\mathcal{O}_{\alpha}= \frac{\partial^2 }{\partial z_{1,...,i-1,i ,X}\partial    z_{1,...,i-1,j,xX}}- \frac{\partial^2 }{\partial z_{1,...,i-1,j ,X}\partial    z_{1,...,i-1,i,xX}}.
$$

Let us associate with it an    {\it  antisymmetrized } GKZ operator  (an A-GKZ operator):
 
  \begin{align}
  \begin{split}
  \label{pagkz}
  &\bar{\mathcal{O}}_{\alpha}:=\frac{\partial^2 }{\partial z_{1,...,i-1,i ,X}\partial    z_{1,...,i-1,j,xX}}- \frac{\partial^2 }{\partial z_{1,...,i-1,j ,X}\partial    z_{1,...,i-1,i,xX}}+\\&+\frac{\partial^2 }{\partial z_{1,...,i-1,x,X}\partial    z_{1,...,i-1,i,jX}},
  \end{split}
  \end{align}
 
 
  \begin{definition}
  	{\it The antysimmetrized GKZ system } (the A-GKZ system)  is the following system of partial differential equations \begin{equation}\label{systema}\bar{\mathcal{O}}_{\alpha}F=0.\end{equation}
  \end{definition}
 Note that in this system the analogues of homogeneity conditions  \eqref{e1} from the system GKZ are omitted.

Introduce functions:
	\begin{equation}
	\label{fgm}
	F_{\gamma}(z):=\me^J=\sum_{s\in\mathbb{Z}_{\geq 0}^k}\frac{(-1)^s}{s!}      J_{\gamma-sr}^s(z),
	\end{equation}

where  $s!=\prod_i s_i!$.   When differentiates the formula  \eqref{fs},  one gets

\begin{equation}
\label{pdf}
\frac{\partial}{\partial z_{X}}J_{\gamma}^s(z)=J_{\gamma-e_{X}}^s(z),\,\,\,\frac{\partial}{\partial z_{X}}F_{\gamma}(z)=F_{\gamma-e_{X}}(z),
\end{equation}

where  $X\subset \{1,...,n\}$ is a proper subset and  $e_{X}$ is a unit vector corresponding to the coordinate    $z_X$.


Let us prove the following statement.

\begin{theorem}
	 \label{agkz}
Consider a set of vectors $\{\gamma_p\}$, $\gamma_p\in\mathbb{Z}^N$, that satisfies the following conditions.

\begin{enumerate}
	\item  Fro every vector  $\gamma_p$ there exists   $b\in B$, such that the vector $\gamma_p+b$ has only non-negative coordinates.
	\item The set   $\{\gamma_p\}$ is maximal set that consists of vectors that satisfy the condition   1 and which is linearly independent in  $\mathbb{Z}^N/B$.
\end{enumerate}
\end{theorem}

The function $F_{\gamma}$ is  called an   {\it irreducible }  solution of the system  \eqref{systema}.  Note that this solution is defined by the vector  $\gamma$, not by the class  $\gamma+B$.
\proof

Before it was proved that the function $F_{\gamma}(z)$  is a solution of the A-GKZ system.

For a monomial  $z^{\gamma}$  the vector of exponents  $\gamma$  is called  {\it a support } of this monomial. {\it A support of a function},  presented as  a sum of a power series,  is set of support of monomials occurring in this series with non-zero coefficients. A support of a function  $F$ is denoted as $supp F$.



Take a solution  $F$.  Represent  $suppF$ as a union of subsets of type $\gamma+B$.
For each such subset take in   $F$ monomials whose supports belong to this subset.  Denote the resulting function as   $F^{\gamma}$. If this function satisfies the system    $\forall \alpha\,\,\,\mathcal{O}_{\alpha}(F^{\gamma})=0$,  then the corresponding support is called extreme (or {\it an extreme point} in
$supp F$). The term  "point"  $\,\,$  is correct since when one considers the support  $mod B$ then it becomes a point.

An irreducible solution  $F_{\delta}$  has a unique extreme point  $\delta+B$.

To prove the Theorem let us first prove the Lemma.

\begin{lemma}
	\label{l2}
Every polynomial solution of the system 	 \eqref{systema}  can be presented as a linear combination of irreducible solutions.
\end{lemma}

\begin{proof} Take an arbitrary solution   $F$ and decompose it into a sum of functions 
	$F^{\gamma}$ with supports   $\gamma+B$.
	
Introduce a partial order on the sets  $\gamma+B$. Put
	\begin{equation}
\label{por1}
	\gamma+B\preceq\delta+B, \text{ if $\gamma+sr=\delta\,\,\, mod B$, $s\in\mathbb{Z}^k_{\geq 0}$.}
	\end{equation}
	

Since only polynomial solutions are considered, there exist summands  $F^{\gamma}$ with supports $\gamma+B$ which are maximal with respect to the order.  Let us prove that these summands are extreme. Indeed one has
	
	$$
	\bar{\mathcal{O}}_{\alpha}F^{\gamma}=\mathcal{O}_{\alpha}F^{\gamma}+\frac{\partial^2}{\partial z^{v_0^{\alpha}}}F^{\gamma}.
	$$
	
	If $suppF^{\gamma}=\gamma+B$, then $supp ( \mathcal{O}_{\alpha}F^{\gamma})=\gamma-v_{\alpha}^++B$,  and $supp(\frac{\partial^2}{\partial z^{v^{\alpha}_0}}F^{\alpha})=\gamma-v^{\alpha}_0+B$. Since $\bar{\mathcal{O}}_{\alpha}F=0$, then by considering the supports one concludes  that the summand  $\mathcal{O}_{\alpha}F^{\gamma}$, in the case when it is non-zero, contracts with one of the expressions  of type   $\frac{\partial^2}{\partial z^{v_0^{\alpha}}}F^{\delta}$ or  $\mathcal{O}_{\alpha}(F^{\delta})$. By considering supports one concludes that  $\mathcal{O}_{\alpha}F^{\gamma}$ can not contract with an expression of the same type for another  $\delta$.  That is why $\mathcal{O}_{\alpha}F^{\gamma}$  contracts with some $\frac{\partial^2}{\partial z^{v^{\alpha}_0}}F^{\delta}$. Then
	$supp F^{\delta}-v_{+}^{\alpha}=\gamma-v^{\alpha}_0$. This means that $supp F^{\delta}=\gamma+v_{+}^{\alpha}-v_0^{\alpha}+B$.  Thus  $supp F^{\delta}\succ \gamma+B$.  But the support   $\gamma+B$ is maximal hence this situation is not possible. Thus $\mathcal{O}_{\alpha}F^{\gamma}=0$.
	
	
	Hence an arbitrary solution has extreme points.
	The corresponding functions  $F^{\gamma}$  have supports of type  $\gamma+B$.  Thus one can write 
	
	$$
	F^{\gamma}=\sum_{t\in\mathbb{Z}^k}c_t\frac{A^{\gamma+tv}}{(\gamma+tv)!}
	$$
for some numeric coefficients  $c_t$.   Since   	$F^{\gamma}$  is non-zero, there exists   $b\in B$, such that the vector  $\gamma+b$ has only non-negative coordinates. Also since	
	 $F^{\gamma}$
is annihilated by the operators   $\mathcal{O}_{\alpha}$, one can conclude that  all $c_t$ are equal. Hence
	$F^{\gamma}$ are   $\Gamma$-series up to multiplication by a constant.

Let us describe  {\bf  a procedure },  that transforms a solution   $V$ of the A-GKZ system into a solution  $W$  of the same system.
	
	\begin{enumerate}
		\item
		Take an extreme point   $\gamma+B$  in $supp V$  and take the corresponding irreducible solution  $F_{\gamma}$.
		\item Subtract from 
		$V$ all the constructed  $F_{\gamma}$ with such a coefficient that the summands  $V^{\gamma}$  (defined by  $V$ by analogy with the definition of   $F^{\gamma}$ for  $F$) in  $V$ with supports  $\gamma+B$ are cancelled. This is possible since both in   $F_{\gamma}$ and in   $V$  the summands with a support in  $\gamma+B$ form a function proportional to a  $\Gamma$-series.
		
		\item The obtained solution is denoted as   $W$.
	\end{enumerate}
 The constructed solution $W$ has the following property: the extreme points in  $suppW$  are strictly lower that the extreme points of    $suppV$ with respect to the order   $\preceq$.

Now let us operate as follows. Take a solution   $F$, apply to it   {\bf the procedure } and obtain a new solution.
Take it's extreme points and apply  {\bf the procedure }  again and so on.
	
Let show that after a finite number of steps this procedure gives a zero function.  To prove it is enough to show that the supports of function that appear on all  the steps are subsets of some finite set.
	
For this purpose for every summand  $F^{\gamma}$ in  $F$  with a maximal support  $\gamma+B$ let us find a set of non-negative integers  $s^{\gamma}_{\alpha}\in\mathbb{Z}^k$, such that
	$
	\gamma-s^{\gamma}r+b:=\gamma-\sum_{\alpha}s^{\gamma}_{\alpha}r^{\alpha}+b
	$  has only non-negative coordinates for some $b\in B$.
	
This set is finite. Indeed, consider a functional $\chi^m_u$  introduced in  \eqref{xii1}.  They are defined by their action on the base vectors  $e_X$ where  $X\subset\{1,...,n\}$, by the following ruler:
	
	\begin{equation}
	\label{xii}
	\chi^m_u(e_X)=\begin{cases} 1,\text{if in  $X$  there are $\geq u$  indices  $\leq m$,} \\ 0 \text{ otherwise }.\end{cases}
	\end{equation}

One can note that   $\chi_u^m(b)=0$ for $b\in B$. For  a vector    $r^{\alpha}$, defined by the formula  \eqref{ra}, one has
	$$
	\chi_{u}^m(r^{\alpha})=\begin{cases}1\text{  for } u=i, j\leq m<y    \\  -1 \text{  for } u=i+1, j\leq m<y \\ 0   \text{    otherwise }. \end{cases}
	$$
	
	Consider first $r^{\alpha}$, such that for them  $i=1$.
When one subtracts from the vector $\gamma$ these vectors  $r^{\alpha}$ with positive coefficients the value of    $\chi_{1}^m$ diminishes. And subtraction of other  $r^{\alpha}$ with bigger value of  $i$ does no affect $\chi_{1}^m$. If then one adds  $b\in B$  then  $\chi_{1}^m$ remains unchanged.  Thus we come to a conclusion:  if one subtracts from  the vector $\gamma$  the vectors  $r^{\alpha}$  with  $i=1$ infinite number of times,  then on some step one obtains a vector such that
	$\chi_{1}^m$ is negative on this vector.  But such a vector can not have only non-negative coordinates since for such vectors  the functional  \eqref{xii} is non-negative.
	
	Then one considered the vectors $r^{\alpha}$, such that  $i=2$   and the functional $\chi_{2}^m$.  One concludes that it is possible to subtract them from  $\gamma$  only finite number of times ans so no.

	
	Introduce a notation:
	
	\begin{equation}
	\label{mgm}
M_{\gamma}=\bigcup\{\gamma-sr+B\},	\end{equation}
	
a union is taken over all obtained above 	  $s^{\gamma}=\{s^{\gamma}_{\alpha}\}$.
	
	One has  $suppF_{\gamma}\subset M_{\gamma}$ since $F_{\gamma}=\sum_{s\in\mathbb{Z}^k_{\geq 0}}\frac{(-1)^sJ_{\gamma-sr}^s}{s!}$ and
	$
	supp J_{\gamma-sr}^s=\gamma-sr+B
	$,   and the function $J_{\gamma-sr}^s$ is non-zero if and only if in the support there are vectors than have only non-negative coordinates.

Also note that if   $\delta\prec\gamma$ then $M_{\delta}\subset M_{\gamma}$.
	
	One has  $supp F\subset \bigcup_{\gamma} M_{\gamma}$ where the union is taken over  all extreme points $\gamma$.
	Indeed, suppose the opposite:  there exists  $\delta\in suppF$, but  $\delta\notin \bigcup_{\gamma} M_{\gamma}$.
	 Consider $F^{\delta}$.    Apply the arguments from the proof of the statement that the maximal points are extreme.  Then one concludes that if   $\mathcal{O}_iF^{\delta}\neq 0$ then  $\delta'=\delta+r_i\in supp F\,\,  mod B$.   Also $\delta\prec\delta'$ and $\delta'\notin \bigcup_{\gamma} M_{\gamma}$
	 since otherwise  the  lower support  $\delta$ also  belongs  to  $\bigcup_{\gamma} M_{\gamma}$.
	Thus  we can increase  (relatively the order  $\prec$) the support adding no points from  $\bigcup_{\gamma} M_{\gamma}$ until we add  a point $\delta''\in supp F$, such that $\mathcal{O}_{\alpha}F^{\delta''}=0$. But this is an extreme point thus it belongs to $\bigcup_{\gamma} M_{\gamma}$. Hence we get a contradiction.
	
Thus on each step of {\bf  the procedure  }  the support belongs to the set  $\bigcup_{\gamma} M_{\gamma}$ where the union is taken over extreme points of the support of the function  $F$. This set is finite. Since on each step the support becomes smaller then after a finite number of steps one obtains an empty set.  This means that we have presented the functions   $F$  as a linear combination of functions  $F_{\gamma}$.

\end{proof}

\begin{lemma}
	\label{l3}
The functions  $F_{\gamma}$  form the formulation of the Theorem  \ref{agkz} are linearly independent.
	
\end{lemma}
\proof

Indeed, let

\begin{equation}\label{in}\sum_pc_pF_{\gamma_p}=0.\end{equation}

Among the sets $\gamma_1+B$,$\gamma_2+B$,... choose a maximal element with respect to the order $\prec$. Consider the corresponding summand  $c_iF_{\gamma_i}$. Take in the expression \eqref{fgm} a summand $\mathcal{F}_{\gamma_i}$. Due to the conditions  1 from the Theorem \ref{agkz}, $\mathcal{F}_{\gamma_i}\neq 0$.   Since   $\gamma_i+B$ is maximal, one obtains that  $\mathcal{F}_{\gamma_i}$ cannot contract with any summand in   \eqref{in}. One obtains a contradiction.

The Theorem  \ref{agkz} is proved.

\endproof

\begin{remark}
The solution $F_{\gamma}$ has the following property. It's support is the set 
	 $M_{\gamma}$ of type \eqref{mgm}.
	 If one   represents this functions as a sum of a series and takes summands with the support  $\gamma+B$ then one obtains  $\mathcal{F}_{\gamma}$.  This solution in some natural sense is the simplest of the A-GKZ system generates by the solutions $\mathcal{F}_{\gamma}$ of system GKZ.
\end{remark}

\section{ The Gelfand-Kapranov-Zelevinsky base}
\label{ra3}
Consider functions  $\mathcal{F}_{\gamma}(z)$ corresponding to shift vectors that are obtained in the following manner. Consider the set of all possible Gelfand-Tsetlin diagrams for an irreducible finite dimensional representation of  $\mathfrak{gl}_n$, construct the corresponding shifted lattices and take the corresponding shift vectors  $\gamma$. Substitute into these functions instead of variables  the  $z_X$  the determinants $a_X$. The resulting function is denoted as   $\mathcal{F}_{\gamma}(a)$.

 \begin{propos}
 	\label{pzh}
 	
 The functions $\mathcal{F}_{\gamma}(a)$  belong to the representations with the highest vector   \eqref{stv}
 \end{propos}

 \proof In the book  \cite{zh} it is proved that a function on the group belongs to the representation with the highest vector \eqref{stv} if and only if the following conditions hold.

 \begin{enumerate}
 	\item $L^{-}f(a)=0$, where  $L^-$  is left infinitesimal shift by  negative root element. Such a shift acts onto row indices of a determinant. If one writes it's action explicitly one sees that for a function of type   $f(a_X)$ this conditions always holds.
 	
 	\item  $L_{i,i} f(a)=m_if(a)$, $i=1,2,...,n$,  where the operators $L_{i,i}$  are left infinitesimal shifts by the elements   $E_{i,i}$.

 	\item $(L_i^+)^{q_i+1}f(a)=0$, $i=2,3,...,n$. The operators  $L_i^+$  are left infinitesimal shifts by positive simple root elements, that is by elements   $E_{i-1,i}$,  and  $q_{i}=m_{i-1}-m_{i}$.

 \end{enumerate}

The conditions   2 and 3  for a polynomial in determinants mean that in each monomial the sum of exponents of determinants of order  $i$ equals  $m_{i}-m_{i-1}$.  This conditions holds for a function  $\mathcal{F}_{\gamma}(a)$ corresponding to the shift vectors described above. Thus it belongs to the representation with the highest vector  \eqref{stv}.

 \endproof

Let us show that these functions form a base that we call  {\it the Gelfand-Kapranov-Zelevinsky base}.
 Let us also find it's relation to the Gelfand-Tsetlin base.

\subsection{The proof the fact that the functions  $\mathcal{F}_{\gamma}(a)$ form a base in a representation}

In the representation with the highest vector \eqref{stv} there are vectors $\mathcal{F}_{\gamma}(a)$ indexed by the Gelfand-Tsetlin diagram. To prove that they form a base it is sufficient to prove that they are linearly independent. If the variables $a_X$  are independent then the proof of the linear independence of the functions  $\mathcal{F}_{\gamma}(a)$ would be very simple. The problem is that the determinants  $a_X$  satisfy the  Plucker relations  (these are all relations between determinants of a square matrix, see \cite{cmb}).

The strategy to overcome this difficulty is the following. We  define a "canonical form"  of  $\mathcal{F}_{\gamma}(a)$ with respect to Plucker relations. Using it we derive that the functions $\mathcal{F}_{\gamma}(a)$ span a representation. Using the irreducibility property we conclude that $\mathcal{F}_{\gamma}(a)$  form a base in the representation with the highest vector \eqref{stv}.

To realize this strategy one  note the following. With a base vector 

$$v_{i,j,x,X}=(...,1_{z_{iX}},...,-1_{z_{jX}},...,-1_{z_{ixX}},...,1_{z_{jxX}},...)$$

one associates a Plucker relation of the following type
\begin{equation}
\label{spl}
a_{1,...,i-1,i,X}a_{1,...,i-i,j,x,X}-a_{1,...,i-1,j,X}a_{1,...,i-i,i,x,X}+a_{1,...,i-1,x,X}a_{1,...,i-i,i,j,X}=0
\end{equation}

Denote an ideal generated by these relations as  $IP$ (we do not discuss the question whether it coincides with the ideal$Pl$, generated by {\it all }  Plucker relations).

Instead of determinants  $a_X$ consider independent variables   $A_X$ (we also use a notations  $A$  for the set of all variables  $A_X$). Introduce a notation. Take a polynomial in variables  $A_X$:

$$
f(A)=\sum_{\beta}c_{\beta}A^{u_{\beta}},\,\,\,\, c_{\beta}\in\mathbb{C},
$$

where  $\beta$  is some index enumeration monomials of this polynomial,  $u_{\beta}$  is vector of exponents of the corresponding  monomial and  $A^{u_{\beta}}$ is a multi-index notation for a monomial.   Relate with the polynomial a differential operator that is obtained by the substitution   $A_X\mapsto \frac{d}{dA_X}$:

$$
f(\frac{d}{dA})=\sum_{\beta}c_{\beta}(\frac{d}{dA})^{u_{\beta}},\,\,\,\, c_{\beta}\in\mathbb{C},
$$

Then
	$$f(a)=0 \,\,\,mod \,\,\, IP$$ if and only if the differential operator 
	$$
	f(\frac{d}{dA})
	$$  acts  as zero on the space of solutions of the A-GKZ system.  

Introduce a notation for this action
	
	\begin{equation}
\label{deist}
	 f(A)\curvearrowright  F(A):=f(\frac{d}{dA})  F(A).
	\end{equation}

 Since one has the base   $F_{\gamma}(A) $  in the space of solutions of the A-GKZ system, one has the following statement

	 \begin{lemma}

	$f(a)=0 \,\,\, mod \,\,\,\, IP$  if and only if 
	$$
f(\frac{d}{dA})\curvearrowright  F_{\gamma}(A)=0.
	$$

	 \end{lemma}
 
 Note that in the formula above  the equality to zero is assumed in ordinary sense, not  $mod IP$.
 
 
Let us find an explicit formula for the action $
 \mathcal{F}_{\delta}(A)\curvearrowright  F_{\gamma}(A)
 $

First of all the following relation between binomial coefficients takes place.

\begin{propos}[\cite{km}]
	$$\binom{N}{t+l}=\sum_{N=N_1+N_2}\binom{N_1}{t}\binom{N_2-1}{l-1}.  $$
\end{propos}




\begin{cor}
	\begin{equation}\label{trt}\binom{(t_i+l_i)+k_i}{t_i+l_i}=\sum_{s_i\in\mathbb{Z}_{\geq 0}}\binom{l_i+s_i-1}{l_i-1}\binom{t_i+k_i-s_i}{t_i-s_i}+...\end{equation}
\end{cor}

Now let us prove the following statement.

\begin{lemma}
	\begin{equation}
	\label{osnf}
	\mathcal{F}_{\gamma}(\frac{d}{dA})F_{\omega}(A)=\sum_{s\in\mathbb{Z}_{\geq 0}^k}\frac{(-1)^s}{s!}J_{\gamma+v}^s(1)F_{\omega-\gamma-sr}(A),
	\end{equation}
	
	 where  $J_{\gamma+v}^s(1)$ is a result of substitution of  $1$ instead of all arguments.
	
\end{lemma}

\proof

Write: $\mathcal{F}_{\gamma}(\frac{d}{dA})=\sum_{l\in\mathbb{Z}^k}\frac{(\frac{d}{d A})^{\gamma+lv}}{(\gamma+lv)!}$. Find an action of the operator  $(\frac{d}{d A})^{\gamma+lv}$ onto the summand $J_{\omega}^p(A)$ from   $F_{\omega}$.  Using the ruler  \eqref{pdf} one obtains

$$
(\frac{d}{d A})^{\gamma+lv}J_{\omega-pr}^p(A)=J_{\omega-\gamma-pr-lv}^p(A).
$$

Now consider in detail $J_{\omega-\gamma-pr-lv}^p(A)$. Let us use a notation

$$
\binom{\tau+p}{p}:=\prod_{i=1}^k \binom{\tau_i+p_i}{p_i}.
$$
One has

\begin{align*}
&\frac{1}{p!}J_{\omega-\gamma-pr-lv}^p(A)=\sum_{\tau\in\mathbb{Z}^k}\frac{\binom{\tau+p}{p}A^{\omega-\gamma-pr-lv+\tau v}}{(\omega-\gamma-pr-lv+\tau v)!}=\\
&=\sum_{t\in\mathbb{Z}^k}\frac{\binom{t+l+p}{p}A^{\omega-\gamma-pr+t v}}{(\omega-\gamma-pr+tv)!}.
\end{align*}

Apply the equality \eqref{trt}. Since $\sum_{t\in\mathbb{Z}^k}\frac{\binom{t+p-s}{p-s}A^{\omega-\gamma-sr+t v}}{(\omega-\gamma-pr+tv)!}=\frac{1}{(p-s)!}J^{p-s}_{\omega-\gamma-pr}(A)$, one has

$$
\frac{1}{p!}J_{\omega-\gamma-(p-s)r-lv}^p(A)=\sum_{s\in\mathbb{Z}^k_{\geq 0}}\binom{l+s-1}{s-1}\frac{1}{(p-s)!}J^{p-s}_{\omega-\gamma-(p-s)r}(A),
$$

where

$$
\binom{l+s-1}{s-1}:=\prod_{i=1}^k\binom{l_i+s_i-1}{s_i-1}.
$$

Now take the expression for  $(\frac{d}{d A})^{\gamma+lv}\frac{1}{p!}J_{\omega-pr}^p(A)$, multiply  it by  $(-1)^p$  
and take a sum over   $p$, one obtains

$$
(\frac{d}{d A})^{\gamma+lv}F_{\omega}(A)=\sum_{s\in\mathbb{Z}^k_{\geq 0}}\binom{l-1+s}{l-1}F_{\omega-\delta-sr}(A)\cdot (-1)^s.
$$

Take a sum over  $l$, one obtains 

\begin{align*}&\mathcal{F}_{\gamma}(\frac{d}{dA})F_{\omega}(A)=\sum_{s\in\mathbb{Z}^k_{\geq 0}}(\sum_l\frac{\binom{l-1+s}{l-1}}{(\gamma+lv)!})F_{\omega-\delta-sr}(A)\cdot (-1)^s=\\
&=\sum_{s\in\mathbb{Z}^k_{\geq 0}}\frac{(-1)^s}{s!}J^s_{\gamma+v}(1)F_{\omega-\delta-sr}(A).
\end{align*}

\endproof

Thus one comes to a conclusion.

\begin{lemma}
	\begin{equation}
	\label{canf}
\mathcal{F}_{\gamma}(A)=\sum_s \frac{1}{s!}J_{\gamma+v}^s(1)A^{\gamma+sr} \,\,\, mod  \,\,\,IP.
	\end{equation}

 \end{lemma}





Note that when one adds to a shift vector the vector  $r^{\alpha}$ then to some row of the Gelfand-Tsetlin diagrams the vector   $[0 \cdots -1 \cdots 1 \cdots 0]$ is added.



Now using these results let us prove that the functions  $\mathcal{F}_{\gamma}(a_Y)$ for the chosen $\gamma$ form the base in the representation.

Write  $E_{i,j}$ as a differential operator

$$
E_{i,j}=\sum_X a_{i,X}\frac{\partial}{\partial a_{j,X}},
$$

where a summation is taken over the subsets  $X\subset\{1,...,n\}$ than do not contain   $i$  and  $j$.

Then

\begin{equation}
\label{eij}
E_{i,j}\mathcal{F}_{\gamma}(a)=\sum_X a_{i,X}\mathcal{F}_{\gamma-e_{j,X}}(a).
\end{equation}


Apply \eqref{canf},  one gets 

 \begin{equation}
 \label{eijf}
E_{i,j}\mathcal{F}_{\gamma}(a)=\sum_X \sum_{s} c_{X,s}\mathcal{F}_{\gamma-e_{j,X}+e_{i,X}+sr}(a),\,\,\, c_{X,s}\in\mathbb{C}.
\end{equation}

Thus it is proved that the span of all  $\mathcal{F}_{\gamma}$  is a representation of the algebra  $\mathfrak{gl}_n$. This representation is generated by vectors indexed by the Gelfand-Tsetlin diagrams and it is contained in the representation with the highest vector  \eqref{stv}. Using the arguments of irreducibility  and dimension one get the following statement.

\begin{theorem}
\label{ipt}
Consider the set of all Gelfand-Tsetlin diagrams $(m_{i,j})$ for an irreducible finite dimensional representation of   $\mathfrak{gl}_n$,  construct a shifted lattice for each diagram.  For each shifted lattice fix a presentation in the form $\gamma+B$ and take the corresponding shift vectors   $\gamma$.  Then the functions  $\mathcal{F}_{\gamma}(a)$ form a base in the representation with the highest vector   \eqref{stv}.

Or, equivalently the functions $\mathcal{F}_{\gamma}(A)$ of the independent variables $A$  form a base in the representations $mod\,\,\,\, IP$.
\end{theorem}

\subsection{A triangular relation between the Gelfand-Tsetlin and the Gelfand - Kapranov-Zelevinsky bases}

Introduce a notation  $$G_{\gamma}(a)$$  for a function corresponding to a Gelfand-Tsetlin base vector, corresponding to a diagram  $(m_{i,j})$,   for which there corresponds a shift vector   $\gamma$ (see Theorem \ref{ipt}).

From the formula  \eqref{eijf} one can derive the following statement.

\begin{theorem}
 The Gelfand-Kapranov-Zelevinsky base  $\mathcal{F}_{\gamma}(a)$ is related to the Gelfand -Tsetlin base $G_{\gamma}(a)$  by  {\it  an upper-triangular } with respect to the order  \eqref{por1} transformation.
\end{theorem}
\proof

Take a Gelfand-Tsetlin diagram   $\gamma=\gamma_1$  and consider a function $G_{\gamma_1}$, corresponding to the diagram  $(m_{i,j})$, and consider a function  $\mathcal{F}_{\gamma_1}$ also corresponding to this diagram.

Let us do the following constructions.

The first step is the following.
 A number  $m_{1,2}-m_{1,1}$ is the maximal power of the operator  $E_{1,2}$ which being applied to   $\mathcal{F}_{\gamma_1}$ does not give zero. This follows from the equation  \eqref{sdr}.  The same is true  for the function  $G_{\gamma_1}$,  this follows fror the formulas for the action of the generators in the Gelfand-Tsetlin base. When one applies  $E_{1,2}^{m_{1,2}-m_{1,1}}$   to  $\mathcal{F}_{\gamma_1}$ and $G_{\gamma_1}$ one obtains functions that we denote as  $\mathcal{F}_{\gamma_2}$   and
$G_{\gamma_2}$. 
They correspond to the diagram  $\gamma_2$ that is obtained form $\gamma_{1}$ by the change

 $$
 m_{1,1}\mapsto m_{1,2}. $$

Now describe the $k$-th step of the construction. We are given a diagram  $\gamma_{k-1}$ which is maximal with respect to  $\mathfrak{gl}_{k-1}$.
Note that $m_{1,k}-m_{1,k-1}$ is the maximal power of   $E_{1,k}$,  which being applied to  $\mathcal{F}_{\gamma_{k-1}}$ or $G_{\gamma_{k-1}}$ gives non-zero functions. The reason is the same as above. As a result of these actions one obtains functions  $\mathcal{F}_{\gamma_{k-1,1}}$, $G_{\gamma_{k-1,1}}$, corresponding to a diagram   $\gamma_{k-1,1}$, which is obtained form $\gamma_{k-1}$ by the change

$$
m_{1,k-1}\mapsto m_{1,k}, ..., m_{1,1} \mapsto m_{1,k}.$$

Then one notes that $m_{2,k}-m_{2,k-1}$ is the maximal power of    $E_{1,k}$ which being applied to  $\mathcal{F}_{\gamma_{k-1,1}}$ or  $G_{\gamma_{k-1,1}}$ gives non-zero functions.  As a result one obtains functions $\mathcal{F}_{\gamma_{k-1,2}}$, $G_{\gamma_{k-1,2}}$, corresponding to a diagram  $\gamma_{k-1,2}$, which is obtained from  $\gamma_{k-1,1}$by the change

$$
m_{2,k-1}\mapsto m_{2,k}, ..., m_{2,2} \mapsto m_{2 ,k}$$

and so on.  After $k-1$ such transformation one obtains a $\mathfrak{gl}_{k}$-highest diagram   $\gamma_{k-1,k-1}$  which we denote as  $\gamma_{k}$.  This is the end of the   $k$-th step.

Finally one gets a  $\mathfrak{gl}_{n-1}$-highest vector for which  (see \cite{1963})   the GKZ vector coincides with the  Gelfand-Tsetlins's  vector, thus one gets  

\begin{equation}
\label{efg}
 E_{n-2,n-1,}^{m_{n-1,n-2,}-m_{n-2,n-2}}...E_{1,2}^{m_{1,2}-m_{1,1}}\mathcal{F}_{\gamma}= E_{n-2,n-1,}^{m_{n-1,n-2,}-m_{n-2,n-2}}...E_{1,2}^{m_{1,2}-m_{1,1}}G_{\gamma}.
\end{equation}

 Now let us "remove" operators in the left hand side and in the right hand side of the equality 
\eqref{efg}. Remove the operator $ E_{n-2,n-1,}^{m_{n-2,n-1}-m_{n-2,n-2}}$. Using the notations introduced above one gets 
 $$
 \mathcal{F}_{\gamma_{n-2,n-1}}=G_{\gamma_{n-2,n-1}}+f,
 $$
  where  for $f$  the following is true:
  \begin{enumerate}
  \item This vector is  $\mathfrak{gl}_{n-2}$-highest with the same highest weight  as   $
  \mathcal{F}_{\gamma_{n-2,n-1}}$ and $G_{\gamma_{n-2,n-1}},
  $  \item  It has the same weight as  $ \mathcal{F}_{\gamma_{n-2,n-1}}$ and $G_{\gamma_{n-2,n-1}}$,   \item $E_{n-2,n-1,}^{m_{n-2,n-1}-m_{n-2,n-2}}f=0$. \end{enumerate}

From  these facts one concludes that that $f$ is a sum of $\mathfrak{gl}_{n-2}$-highest vectors corresponding to the Gelfand-Tsetlin diagrams such that for them the rows $n$ and  $n-2$ coincide with the rows  of the diagram corresponding to the shift vector $\gamma_{n-2,n-1}$,  and the row   $(n-1)$  is obtained from the row of the diagram corresponding to   $\gamma_{n-2,n-1}$ by adding vectors of type   $(....-1.......+1)$. But  adding of such  vectors to the row of a diagram  is equivalent to the adding   of vectors   $r_i$ to  the shift vector corresponding to the diagram. That is   $f$ is a sum of functions of type  $G_{\gamma_{n-2,n-1}+sr}$, where  $sr:=s_1r_1+...+s_kr_k$ for some  $s\in\mathbb{Z}_{\geq 0}^k$.

One has  $G_{\gamma_{n-2,n-1}}=E_{n-1,n-3}^{m_{n-3,n-1}-m_{n-3,n-2}}...E_{2,1}^{m_{1,2}-m_{1,1}}G_{\gamma}$, since the diagram corresponding to $\gamma_{n-2,n-1}+sr$  differs  for the diagram corresponding to $\gamma_{n-2,n-1}$  only in the row  $n-1$, then
$$G_{\gamma_{n-2,n-1}+sr}=E_{n-1,n-3}^{m_{n-3,n-1}-m_{n-3,n-2}}...E_{2,1}^{m_{1,2}-m_{1,1}} G_{\gamma+sr}$$

Thus when one removes in   \eqref{efg}  the operator  $E_{n-2,n-1}^{m_{n-2,n-1}-m_{n-2,n-2}}$ one obtains

 \begin{align}
 \begin{split}
 \label{efg1}
&E_{n-3,n-1,}^{m_{n-3,n-1}-m_{n-3,n-2}}...E_{1,2}^{m_{1,2}-m_{1,1}}\mathcal{F}_{\gamma}= \\&= E_{n-3,n-1,}^{m_{n-3,n-1}-m_{n-3,n-2}}...E_{1,2}^{m_{1,2}-m_{1,1}}(G_{\gamma}+\sum_{j}const_j\cdot G_{\gamma+s^jr}),
 \end{split}
 \end{align}

 where $s^j\in\mathbb{Z}_{\geq 0}^k$, and adding  of $s^{j}r$ to the shift vector corresponds to the change of the row   $(n-1)$. Continuing this process one comes to the conclusion that   $ \mathcal{F}_{\gamma}$ is expressed through $G_{\gamma}$   by an upper-triangular transformation. Hence $\mathcal{F}_{\gamma}$ is also related to  $G_{\gamma}$ by an upper-triangular transformation.
 \endproof

 \begin{remark}
 It is known that the function  $G_{\gamma}$, corresponding to a Gelfand-Tsetlin diagram, can be obtained form the highest vector  by application of lowering operators  to the highest vector    $v_0$ (see  \cite{zh})

 \begin{align*}
 G_{\gamma}=\prod_{k=2}^n\prod_{i=k-1}^1\nabla_{k,i}^{m_{k,i}-m_{k-1,i}} v_0.
 \end{align*}

The vectors of the Gelfand-Kapranov-Zelevinsky base can be obtained in the same manner but one needs to use the following lowering operators:

 $$\widetilde{\nabla}_{k,i}=a_{1,...,i-1,k}\frac{\partial}{\partial a_{1,...,i-i,i}}.
 $$

 In some sense $\widetilde{\nabla}_{k,i}$ is a simplification of the operators  $\nabla_{k,i}$ (see the formula for $\nabla_{k,i}$ in \cite{zh}).
 \end{remark}
\section{A function corresponding to a diagram}

In the case  $n\geq 4$ the functions $\mathcal{F}_{\gamma}(a)$ are not the Gelfand-Tsetlin vectors. To find a function  that is the Gelfand-Tsetlin vector let us consider  irreducible solutions $F_{\gamma}(a)$. 




Let us prove the following statement.

\begin{lemma} Consider shift vectors $\gamma$  corresponding to all possible Gelfand-Tsetlin diagrams of an irreducible finite-dimensional representation of $\mathfrak{gl}_n$. Then the functions  $F_{\gamma}(A)$ of independent variables  $A_X$  form a base in the representation of the algebra  $\mathfrak{gl}_n$ with the highest vector \eqref{stv} (in which we change $a_X\mapsto A_X$). The generators of the algebra act by the ruler  \eqref{eij}.
\end{lemma}
In other words the functions   $F_{\gamma}(A)$  span a representation even without usage of relations between determinants.

\begin{proof}
	
Consider the ideal  $Pl$  of all relations between the determinants  $a_X$.  It is an ideal in the ring of polynomials in independent variables  $A_X$. The action of $E_{i,j}$ preserves this ideal.  Note that  $IP\subset Pl$.

To every polynomial in variables   $A_X$ there corresponds a differential operator that is obtained by the substitution  $A_X\mapsto \frac{d}{d A_X}$.  To the ideals  $IP\subset Pl$ in the polynomial ring there correspond ideals  $D_{IP}\subset D_{Pl}$ in the ring of differential operators with constant coefficients.

Consider the spaces of polynomial solutions for these ideals that is the spaces of functions that are annihilated by all operators from the  corresponding ideal. One has an inclusion 

$$Sol_{D_{IP}}\supset Sol_{ D_{Pl}}.$$

Since the action of  $E_{i,j}$ preserves the ideal  $Pl$ then the space $Sol_{ D_{Pl}}$ is invariant under the action of    $\mathfrak{gl}_n$.

Consider the highest weight  $[m_1,...,m_n]$ and take finite-dimensional linear spaces 

\begin{equation}\label{solsol}Sol^{m_1,...,m_n}_{D_{IP}}\supset Sol^{m_1,...,m_n}_{ D_{Pl}},\end{equation}

of solutions such that the sums of exponents for $A_X$ such that   $|X|=i$, equals  $m_i-m_{i-1}$. Let us do the following observations.

{\bf 1.} Since the action of  $E_{i,j}$  preserves these sums of exponents then the space on the right hand side in  \eqref{solsol} is invariant under the action of  $E_{i,j}$.

{\bf 2.} Both spaces in  \eqref{solsol} contain a monomial \eqref{stv} (in which one changes  $a_X\mapsto A_X$). Indeed consider any basic Plucker relation (in particular a relation  \eqref{spl}). Every summand in the basic Plucker relation {\it does not }  contain two variables from \eqref{stv}. Transform a  relations to a differential operator. Then each  summand in this operator annihilates  \eqref{stv}.

As a corollary of  {\bf 1,2} one gets that the space on the right hand side of  \eqref{solsol} contains an irreducible representation with the highest weight  $[m_1,...,m_n]$. Hence  it's dimension is greater or equal then the dimention of this representation.

{\bf 3.} Due to the Theorem   \ref{agkz} in   $Sol^{m_1,...,m_n}_{D_{IP}}$ there exists a base of type  $\{F_{\gamma_{p}}(A)\}$, where the vectors  $\gamma_{p}$ form a maximal linear independent   $mod B$ subset of vectors  in the set of all vectors  $\gamma\in\mathbb{Z}^N$ such that:  1)  all the coordinates become non-zero after adding of a vector from $B$, 2) $\sum_{X:|X|=i}\gamma_X=m_i-m_{i-1}$. The number of  vectors   $\gamma_{p}$ is equal to the number of independent  $mod B$ solutions of the system from the definition  \ref{sdr},  that are constructed from all possible Gelfand-Tsetlin diagrams  $(m_{i,j})$ with a  fixed upper row $[m_1,...,m_n]$.

Hence the dimension  $Sol^{m_1,...,m_n}_{D_{IP}}$  equals to the dimension of the irreducible representation with the highest weight  $[m_1,...,m_n]$.
Also one sees that the basic  $\gamma_p$ are the shift vectors corresponding to all possible Gelfand-Tsetlin diagrams as it is stated in the formulation of the Lemma.

Using the  conclusion from   {\bf 1,2}  and the conclusion from   {\bf 3} one obtains  that the dimension of the space on the right side in  \eqref{solsol}  is greater or equal than the dimension of the space on the left hand side in \eqref{solsol}. Thus one has 

$$Sol^{m_1,...,m_n}_{D_{IP}}= Sol^{m_1,...,m_n}_{ D_{Pl}}$$

As it was pointed in   {\bf 3},  the span of functions   $<F_{\gamma}(A)>$, listed in the formulation is the space   $Sol^{m_1,...,m_n}_{D_{IP}}$. This is a representation of the algebra   $\mathfrak{gl}_n$ since such property has $ Sol^{m_1,...,m_n}_{ D_{Pl}}$.   Also $<F_{\gamma}(A)>$ contains \eqref{stv}, it has the same dimension as the irreducible representation generated by \eqref{stv}. Hence   $<F_{\gamma}(A)>$ coincides with this representation.

\end{proof}

As a corollary one gets that functions on the group  $F_{\gamma}(a)$ also form a representation. It contains the highest vector   \eqref{stv}. Hence if one  takes shift vectors corresponding to different Gelfand-Tsetlin diagrams,  one obtains functions   $F_{\gamma}(a)$  that form a base in a representation. Due to the Theorem  \ref{ipt} the bases  $F_{\gamma}(a)$ and $\mathcal{F}_{\gamma}(a)$ are related by an invertible linear transformation  $mod IP$.

Below we prove that the base  $F_{\gamma}(a)$ is related to the Gelfand-Tsetlin base by a  {\it  low-triangular transformation} relatively to the order  \eqref{por1}.

To prove this fact and to find this transformation explicitly we use an invariant scalar product. The Gelfand-Tsetlin base is orthogonal relatively this scalar product.
Thus the transformation from the base   $F_{\gamma}(a)$ to the Gelfand-Tsetlin base is a lower-triangular transformation that diagonalizes  the quadratic form of the scalar product.

 \subsection{ An invariant scalar product in the functional representation.}
 
 \label{isk}

On a finite dimensional irreducible representation of  $GL_n$ there exists a unique up to multiplication on a constant hermitian product $\{.,.\}$ which is  invariant under the action of $U_n$.  It defines an invariant $\mathbb{C}$-bilinear scalar product  $(x,y):=\{x,\bar{y}\}$.
 
  The invariance of    $\{.,.\}$ with  respect to the action of the group    $U_n\subset GL_n(\mathbb{C})$ means that the following equality for the  $\mathbb{C}$-bilinear scalar product 
 
   \begin{equation}
   \label{inv}
   (E_{i,j}v,w)=(v,E_{j,i}w).
   \end{equation}

 Thus on an irreducible finite dimensional representation of $GL_n$ there exists a unique invariant  $\mathbb{C}$-bilinear scalar product for which \eqref{inv} holds.
Let us find it in terms of functional realization.

Consider firstly the space $V$  spanned by independent variables   $A_p$, $p=1,...,n$, onto which the algebra   $\mathfrak{gl}_n$ acts by the ruler

 $$
 E_{i,j} A_p=\delta_{j,p}A_i,
 $$

 where  $\delta_{j,p}  $ is the Kronecker  symbol.

 Introduce a $\mathbb{C}$-bilinear scalar product  $<,>$:

   $$
   <A_p,A_q>=\delta_{p,q}
   $$

    Let us check that it is invariant. Indeed:

  $$
 < E_{i,j}A_p,A_q>=\delta_{j,p}\delta_{i,q},\,\,\,\, <A_p,E_{j,i}A_q>=\delta_{i,q}\delta_{p,j}.
  $$

 Thus it is invariant. 

 Now consider the following construction. Let   $V$ be a space of representation of  $\mathfrak{gl}_n$ with an invariant scalar product   $<,>$. Then on   $V^{\otimes n}$ there exists an invariant scalar product given by the ruler

   $$
   <v_{i_1}\otimes...\otimes v_{i_n},w_{j_1}\otimes...\otimes w_{j_n}>=<v_{i_1},w_{j_1}>...<v_{i_n},w_{j_n}>.
   $$

 There exists a projection
$$
\pi  : V^{\otimes n} \rightarrow Sym^n(V),
$$

which is agreed with the action of  $\mathfrak{gl}_n$. It has a right inverse

$$
\pi^{-1}:   Sym^n(V)\rightarrow   V^{\otimes n} ,\,\,\,\,\,\,  v_{i_1}\cdot...\cdot v_{i_n}\mapsto \frac{1}{n!} \sum_{\sigma\in S_n} v_{i_{\sigma(1)}}\otimes...\otimes  v_{i_{\sigma(n)}}.
$$

It agrees with the action of $\mathfrak{gl}_n$. Thus we have an invariant scalar product given by the ruler

 \begin{equation}
 \label{vn}
 <v_{i_1}\cdot...\cdot v_{i_n},w_{j_1}\cdot...\cdot w_{j_n}>:=<\pi^{-1}(v_{i_1}\cdot...\cdot v_{i_n}),\pi^{-1}(
 w_{j_1}\cdot...\cdot w_{j_n})>.
 \end{equation}

Return to the space   $V$, spanned by variables   $A_p$,  apply to it the ruler  \eqref{vn}.  One obtains that the monomials   $A^{\gamma}$ (the multi-index notation is used) are orthogonal.  The scalar product equals 

\begin{align}
\begin{split}
\label{skp}
 & <A^{\gamma},A^{\gamma}>=\gamma!,\text{ where  } (\gamma_1,...,\gamma_N)!=\gamma_1!...\gamma_N!,\\
 &  <A^{\gamma},A^{\delta}>=0\text{ for }\gamma\neq \delta.
  \end{split}
 \end{align}

Note that the scalar product   \eqref{skp} can be written as follows:

\begin{equation}
 <A^{\gamma},A^{\delta}>=A^{\gamma}\curvearrowright A^{\delta}\mid_{A=0},
\end{equation}

where  $$A^{\gamma}\curvearrowright A^{\delta}:=(\frac{d}{dA})^{\gamma}A^{\delta}.$$

Let us construct an invariant scalar product on the space of functions on the group $GL_n$,  that form an irreducible finite dimensional representation.  The functions are written as expressions depending on determinants. The difficulty is that the determinants are not independent variables, they satisfy the Plucker relations. 

To overcome this difficulty note that  $V=span(\mathcal{F}_{\gamma}(a))=span(F_{\gamma}(a)),$ where  $\gamma$ are all possible Gelfand-Tsetlin diagrams for one representation. Both $\mathcal{F}_{\gamma}(a)$ and $F_{\gamma}(a)$ are bases.

It is enough to define a scalar product between base vectors. Put 

\begin{equation}
\label{skp1}
(F_{\gamma}(a),F_{\delta}(a)):=<F_{\gamma}(A),F_{\delta}(A)>=F_{\gamma}\curvearrowright F_{\delta}\mid_{A=0}.
\end{equation}

\begin{propos}
The formula \eqref{skp1} defines an invariant scalar product in the representation with the highest vector \eqref{stv}.
\end{propos}
\proof
This definition is correct since the functions $F_{\gamma}(A)$ of independent variables  $A$ span a representation (without usage of equivalence   $mod IP$). Since  $<,>$ is invariant then the obtained scalar product is invariant.

\endproof

\subsection{
	A relation between the base $F_{\gamma}(a)$ and the Gelfand-Tsetlin base}

\label{fgc}

Note that in general  $(\mathcal{F}_{\gamma}(a),\mathcal{F}_{\delta}(a))\neq  <\mathcal{F}_{\gamma}(a),\mathcal{F}_{\delta}(a)>=\mathcal{F}_{\gamma}\curvearrowright \mathcal{F}_{\delta}\mid_{A=0}$. To find  $<\mathcal{F}_{\omega}(a),\mathcal{F}_{\delta}(a)>$ we has to express $\mathcal{F}_{\omega}(A)$, $\mathcal{F}_{\delta}(A)$  through  $F_{\omega}(A) $ modulo  $IP$.

\begin{propos} $<pl,F_{\delta}>=0$, where  $pl\in IP$.\end{propos}

\proof   One has $<pl,F_{\delta}>=pl\curvearrowright F_{\delta}\mid_{A=0}$. But the generators of the ideal  $IP$ act onto  $F_{\delta}$  as zero (since $F_{\delta}$  is a solution of A-GKZ).
\endproof
Thus one has

$$
(\mathcal{F}_{\gamma}(a),F_{\delta}(a))=  <\mathcal{F}_{\gamma}(a),F_{\delta}(a)>.
$$

Earlier the order   \eqref{por1}  on the sets  $\gamma+B$ was defined.  One can consider the order   \eqref{por1} being defined on the shift vectors.

Let us prove the following statement.
\begin{propos}
$(\mathcal{F}_{\gamma}(a),F_{\delta}(a))=
\begin{cases} 0\text{ if } \delta \prec \gamma \\
\frac{(-1)^u}{u!}J^u_{\delta}(1)\text{ if }\delta+ur=\gamma,\,\,u\in\,\mathbb{Z}^k_{\geq 0}\end{cases}
$
\end{propos}

\proof

One has to calculate  $\mathcal{F}_{\gamma}\curvearrowright F_{\delta}\mid_{A=0}$. In order to obtain a non-zero result among the summands in  $\mathcal{F}_{\gamma}\curvearrowright F_{\delta}$ a constant must occur. Since $mod B_{}$, one has $supp\mathcal{F}_{\gamma}=\gamma$, and $supp F_{\delta}=\bigcup_{s\in\mathbb{Z}_{\geq 0}^k}(\gamma-sr)$, the result is non-zero if $\delta+ur=\gamma\,\,\, mod\,\,\, B_{}$.

Let  $\delta+ur=\gamma$.  In this case the constant that appears under the action   $\mathcal{F}_{\gamma}\curvearrowright F_{\delta}$ equals

$$
\sum_{t\in\mathbb{Z}^k} \frac{(-1)^u}{u!}\frac{(t+1)...(t+u)}{(\delta+tv)!}=\frac{(-1)^u}{u!}J^u_{\delta}(1)
$$

\endproof

\begin{cor}
The base $F_{\gamma}(a)$ is related with the Gelfand-Tsetlin base by a lower-triangular transformation.
\end{cor}

\proof
From one hand $\mathcal{F}_{\gamma}$ is related with the Gelfand-Tsetlin base vectors by an upper-triangular transformation. From the other hand the scalar product  $(\mathcal{F}_{\gamma}(a),F_{\delta}(a))$ is non-zero only in the case  $\delta\preceq \gamma$. Since the Gelfand-Tsetlin base is orthogonal one obtains the statement.

\endproof

\subsection{The scalar products of functions $F_{\gamma}(a)$}
\label{kvf}

Find the scalar product $(F_{\gamma},F_{\omega})=<F_{\gamma},F_{\omega}>$.   Considers the supports of functions, one concludes that this scalar product is non-zero in the case $\gamma=\delta+l_1r\,\,\, mod B$, $\omega=\delta+l_2r\,\,\, mod B$, $l_1,l_2\in\mathbb{Z}^k_{\geq 0}$.


Also by considering the supports  and using the  expression $F_{\gamma}\curvearrowright F_{\delta}\mid_{A=0}$
one concludes that $
<F_{\delta+l_1r},F_{\delta+l_2r}>,\,\,\,l_1,l_2\in\mathbb{Z}^k_{\geq 0}
$ equals

$$
\sum_{u\in \mathbb{Z}^k_{\geq 0}}(-1)^{l_1+l_2} \frac{(t+1)...(t+u+l_1)(t+1)...(t+u+l_2)}{(\delta-ur+tv)!(u+l_1)!(u+l_2)!}
$$
for which $min(l_1,l_2)=\{min(l_1^i,l_2^i)\}=0$.
Introduce functions

\begin{equation}
\label{juu}
J_{\delta}^{u+l_1;u+l_2}(A):=\sum_{t\in\mathbb{Z}^k}\frac{(t+1)...(t+u+l_1)(t+1)...(t+u+l_2)}{(\delta+tv)!(u+l_1)!(u+l_2)!}A^{\delta+tv}.
\end{equation}

%

   Also introduce functions
\begin{equation}
\label{fll}
F_{\delta}^{l_1,l_2}(A):=\sum_{u\in\mathbb{Z}^k_{\geq 0}}\frac{(-1)^{l_1+l_2}J_{\delta-ur}^{u+l_1;u+l_2}(A)}{(u+l_1)!(u+l_2)!}.
\end{equation}

One has

$$
<F_{\delta+l_1r},F_{\delta+l_2r}>=F_{\delta}^{l_1,l_2}(1),
$$

where it is suggested that $min(l_1,l_2)=0$.

\subsection{A function corresponding to a Gelfand-Tsetlin vector  }
\label{itg}
\begin{definition}
Let  $\delta_0$ be a shift vector which is
 minimal with   respect to  the order   \eqref{por1}.
\end{definition}

 Then one can suppose that  the vector corresponding to an arbitrary shift vector can be written as follows 

$$
\delta_0+mr,\,\,\,\,\, m\in\mathbb{Z}^k_{\geq 0}.
$$

This suggestion implies that the scalar product  $<F_{\delta},F_{\omega}>$ is non-zero only  if the following relation takes  place $\gamma=\delta+l_1r$, $\omega=\delta+l_2r$, $l_1,l_2\in\mathbb{Z}^k_{\geq 0}$.  In contrast to the previous Section this relations for  $\gamma$ and $\omega$ holds exactly not   $modB$.

Consider the bilinear form corresponding to the scalar product in the base    $F_{\gamma}$:

$$
q=\sum_{\delta,\gamma} x_{\gamma}y_{\delta}<F_{\delta},F_{\omega}>,$$

 the summation is taken over all  chosen shift vector.

Let us write explicitly the low-triangular change of coordinates that diagonalizes this quadratic form. Fix an arbitrary   $\delta$   and consider the space $span(F_{\gamma},\gamma\preceq \delta)$. Then the summands    $q$,  that contain  $x_{\delta}$ are written as follows 

  $$
  q=F_{\delta}^{0,0}x_{\delta}^2+\sum_{l\in\mathbb{Z}^k_{\geq 0},\,\,\, l\neq 0}2F_{\delta-lr}^{l,0}(1)x_{\delta-lr}x_{\delta}+\cdots
  $$

Apply the  Lagrange algorithm,  one obtains that the diagonalizing change of variables looks as follows 
\begin{equation}
\label{zamena}
x'_{\delta}=\sum_{l\in\mathbb{Z}^k_{\geq 0}}F_{\delta-lr}^{l,0}(1)x_{\delta-lr}.
\end{equation}

Since  $\delta$  is arbitrary one comes to the following conclusion. Let  $G_{\delta}$ be a function that corresponds to a Gelfand-Tsetlin diagram (corresponding to the shift vector $\delta$).
One has an expression of the diagonalizing variables though the initial  variables. Thus one has an expression  of  $F_{\delta}$ though $G_{\delta}$.

\begin{theorem}
	\label{gt1}
\begin{align}
\begin{split}
\label{fg}
&F_{\delta}(A)=\sum_{l\in\mathbb{Z}^k_{\geq 0}}C_{\delta}^l\cdot G_{\delta-lr}(A),\\
&C_{\delta}^l=F_{\delta-lr}^{l,0}(1)=\sum_{u\in\mathbb{Z}^k_{\geq 0},t\in\mathbb{Z}^k}\frac{(-1)^{l}(t+1)...(t+u+l)(t+1)...(t+u)}{(\delta-(l+u)r+tv)!(u+l)!u!}.
\end{split}
\end{align}

\end{theorem}

One can find an inverse to   \eqref{fg}. The result is the following

\begin{theorem}
	\label{gt2}
\begin{align}
\begin{split}
\label{gf}
&G_{\delta}(A)=\sum_{l\in\mathbb{Z}^k_{\geq 0}}S_{\delta}^l\cdot F_{\delta-lr}(A),\\
&S_{\delta}^0=\frac{1}{C_{\delta}^0},\,\,\,\,S_{\delta}^l=-\frac{C_{\delta}^l}{C_{\delta}^0 C_{\delta-lr}^0},\,\,l\neq 0.
\end{split}
\end{align}

\end{theorem}

\begin{remark}
Unfortunately we don not  know how one show that from the formula  \eqref{gf}  in the case  $n=3$ it follows that $G_{\gamma}=\mathcal{F}_{\gamma}$  modulo the Plucker relations. Mention that the case $n=3$ is very specific. For example using the formula (65)  from \cite{A1},  it is possible to construct another base in the space of solutions of the  A-GKZ system:

 $$
 \tilde{F}_{\gamma}(a)=a_3^{\gamma_3}a_{1,2}^{\gamma_{1,2}}\sum_{s\in\mathbb{Z}_{\geq 0}}const_s\cdot (a_1a_{2,3}-a_2a_{1,3})^s\mathcal{F}_{\gamma-sv^+}(a),
 $$

 This base does not have an analogue in the case  $n>3$.  Using it instead of $F_{\delta}$ one can conclude that  in the case  $n=3$  one has  $G_{\gamma}=\mathcal{F}_{\gamma}$  modulo the Plucker relations.
\end{remark}

\subsection{The coefficients in the Theorems  \ref{gt1}, \ref{gt2}}

\label{cft}
In the formulas  \eqref{fg}, \eqref{gf} there participate the  expressions  $c_l=F_{\delta-lr}^{l,0}(1)$.  This expression is defined as a sum of a series  (actually the sum is finite).  In this  Section we consider this series in detail. 

\subsubsection{ An expression for $J_{\delta}^{u+l_1;u+l_2}(A)$}

First of all let us show that the function $J_{\delta}^{u+l_1;u+l_2}(A)$  can be expressed through simpler functions  $J_{\delta}^{v}$.

Consider the expression 
$
c_t^{a,b}:=(t+1)....(t+a)(t+1)...(t+b)
$
and represent it as a linear combination of expressions $c_t^c:=(t+1)...(t+c)$.  One has an equality 

$$
c_t^{a,b}=\sum_{c=0}^{a+b} k_c c_{t}^{c},
$$

 Let us find  coefficients  $k_c$ in this equality.  We use the operator    $O$ defined as follows

 $$
 Oc_t:=c_t-c_{t-1},
 $$

 and the operator of substitution  $\mid_{t=-1}$.  Let us use the rulers

 \begin{align*}
 &O c_t^{a,b}=a\cdot c_t^{a-1,b}+b\cdot c_{t}^{a,b-1}-ab\cdot c_t^{a-1,b-1},\\
 &O c_t^c=c\cdot c_t^{c-1}.
 \end{align*}

Note that  $c_{t}^{a,b}\mid_{t=-1}\neq 0$ only if    $a=b=0$ and    $c_{t}^{0,0}\mid_{t=-1}=1$, analogously   $c_{t}^{c}\mid_{t=-1}\neq 0$  only if    $c=0$ and    $c_{t}^{0}\mid_{t=-1}=1$.

Using these facts one obtains that  $O^pc_{t}^{a,b}\mid_{t=-1}\neq 0$ under the condition  $max(a,b)\leq p\leq a+b$, and in this case
$O^pc_{t}^{a,b}\mid_{t=-1}=a!b!(-1)^{a+b-p}$.  From the other hand   $O^pc_{t}^{c}\mid_{t=-1}\neq 0$  for  $c=p$,   under the condition   $O^cc_{t}^{c}=c!$.  One gets

$$
k_c=\frac{a!b!}{c!}(-1)^{a+b-c}.
$$

Apply this relation to the function   $J_{\delta}^{u+l_1,u+l_2}(A)$ defined by the formula  \eqref{juu}, one gets

$$
J_{\delta}^{u+l_1,u+l_2}(A)=\sum_c\frac{(u+l_1)!(u+l_2)!}{c!}(-1)^{2u+l_l+l_2-c}J_{\delta}^c(A),
$$
where  $max(l_1,l_2)\leq c\leq 2u+l_1+l_2$.

 For the functions  $F_{\delta}^{l_1,l_2}(A)$,  defined by the equality \eqref{fll}, one has
 \begin{align*}
 F_{\delta}^{l_1,l_2}(A)=\sum_{u\in\mathbb{Z}_{\geq 0}^k}\sum_{u+max(l_1,l_2)\leq c\leq 2u+l_1+l_2}\frac{(-1)^cJ_{\delta}^c(A)}{c!}
 \end{align*}

 Consider the scalar product  $<F_{\delta},F_{\delta}>=F_{\delta}^{0,0}(1)$.  It is written as a sum of values at  $A=1$ of the following function

\begin{align*}
& \frac{1}{0!}J_{\delta}^0(A)\\
& -\frac{1}{1!}J_{\delta-r^{\alpha}}^{e_{\alpha}}(A) & \frac{1}{2!}J_{\delta-r^{\alpha}}^2(A)\\
& \frac{1}{2!}J_{\delta-2r^{\alpha}}^{2e_{\alpha}}(A) & -\frac{1}{3!}J_{\delta-3r^{\alpha}}^{3 e_{\alpha}}(A) && \frac{1}{4!}J_{\delta-4r^{\alpha}}^{4e_{\alpha}}(A)\\
\end{align*}

Here the index  is fixed $\alpha$, one has to consider shift for all indices    $\alpha=1,...,k$.

Consider the scalar product $<F_{\delta},F_{\delta-r_{\alpha}}>=F_{\delta-r_{\alpha}}^{1_{\alpha},0}(1)$.  It is presented as a sum of values at  $A=1$ of the following functions

\begin{align*}
& -\frac{1}{1!}J_{\delta-r^{\alpha}}^{e_{\alpha}}(A)\\
& \frac{1}{2!}J_{\delta-2r^{\alpha}}^{2e_{\alpha}}(A) & -\frac{1}{3!}J_{\delta-2r^{\alpha}}^{3e_{\alpha}}(A)\\
& -\frac{1}{3!}J_{\delta-3r^{\alpha}}^{3e_{\alpha}}(A) & \frac{1}{4!}J_{\delta-3r^{\alpha}}^{4e_{\alpha}}(A) && -\frac{1}{5!}J_{\delta-3r^{\alpha}}^{5e_{\alpha}}(A)\\
\end{align*}

Thus one gets

\begin{align}
\begin{split}
\label{hrn0}
&F_{\delta-lr}^{l,0}(A)=\sum_{u\in \mathbb{Z}^k_{\geq 0}}\sum_{0\leq s\leq u}(-1)^{l+u}\frac{J^{l+u+s}_{\delta-(l+u)r}(A)}{(l+u+s)!}
\end{split}
\end{align}

\subsubsection{A relation to the Horn's functions}
The functions    $J_{\gamma}^s(A)$ have the following description on the language of Horn's functions. Let $\zeta\in\mathbb{C}^k$.   Then  $$H(\zeta)=\sum_{t\in \mathbb{Z}^k}c(t)\zeta^t,\,\,\,c(t)\in\mathbb{C}$$ is called the Horn series if 

$$\frac{c(t+e_{\alpha})}{c(t)}$$  is a rational function of  $t$,  in other words if

$$\frac{c(t+e_{\alpha})}{c(t)}=\frac{P_{\alpha}(t)}{Q_{\alpha}(t)},\,\,\,\alpha=1,...,k,$$
where $P_{\alpha},Q_{\alpha}$ are polynomials  (see \cite{sts}, \cite{GG})).

With a   $\Gamma$-series $\mathcal{F}_{\gamma}(A)$ one associates the following Horn series. Let us write  

\begin{align*}
&\mathcal{F}_{\gamma}(A)=\sum_{x=\gamma+t_1v_1+...+t_kv_k}\frac{A^x}{x!}=\sum_{t}\frac{A^{\gamma+tv}}{(\gamma+tv)!}=A^{\gamma}\sum \frac{(A^{v})^{t}}{(\gamma+t_1v_1+...+t_kv_kN)!}=\\
&=A^{\gamma}\sum_{t}\frac{\zeta^t}{(\gamma+tv)!}=A^{\gamma}H_{\gamma}(\zeta),
\end{align*}

where

$$
\zeta_1=A^{v_1},...,\zeta_k=A^{v_N}.
$$

Then

\begin{align*}
&J^{s}_{\gamma}(A)=\sum_{t}\frac{(t+1)...(t+s)A^{\gamma+tv}}{(\gamma+tv)!}=A^{\gamma}\sum_{t}\frac{(t+1)...(t+s)\zeta^{t}}{(\gamma+tv)!}=\\
&=A^{\gamma}(\frac{d}{d\zeta})^s(\zeta^sH_{\gamma}(\zeta)).
\end{align*}

Thus one obtains that 

\begin{align}
\begin{split}
\label{hrn}
&F_{\delta-lr}^{l,0}(1)=\sum_{u\in \mathbb{Z}^k_{\geq 0}}\sum_{0\leq s\leq u}\frac{(-1)^{l+u}}{(l+u+s)!}  (\frac{d}{d\zeta})^s(\zeta^{l+u+s}H_{\delta-(l+u)r}(\zeta))\mid_{\zeta=1}
\end{split}
\end{align}

\begin{remark}
In the formula  \eqref{hrn} one represents   $F_{\delta-lr}^{l}(1)$  as a value at $1$  of a sum of derivatives of a Horn's function. One can also note  that such an expression is also a horn's function. Thus  $F_{\delta-lr}^{l}(1)$ is a value at one of a function of the hypergeometric type that is this is a hypergeometric constant. 
\end{remark}

\end{fulltext}

\end{document}